\documentclass[11pt, a4paper]{amsart}
\usepackage{amsmath}
\usepackage{amssymb}
\usepackage{color}
\usepackage[utf8]{inputenc}
\usepackage{hyperref}

\def\RR{\mathbb{R}}

\def\div{\mathrm{div}\,}
\def\supp{\mathrm{supp}\,}

\def\to{\rightarrow}
\def\tto{\longrightarrow}

\def\vphi{\varphi}
\def\pa{\partial}
\def\na{\nabla}

\newtheorem{theorem}{Theorem}[section]
\newtheorem{lemma}[theorem]{Lemma}
\newtheorem{proposition}[theorem]{Proposition}
\newtheorem{definition}[theorem]{Definition}

\newtheorem{rk&ex}[theorem]{Remarks \& Examples}
\newtheorem{corollary}[theorem]{Corollary}

\title{A Hele-Shaw problem for tumor growth}

\author{Antoine Mellet \and Beno\^{\i}t Perthame
\and Fernando Quir\'os}

\address{Antoine Mellet\hfill\break\indent Department of Mathematics, University of Maryland, College Park, Maryland 20742\hfill\break\indent  Fondation Sciences Math\'ematiques de Paris, 11 rue Pierre et Marie Curie, 75231 PARIS.}
\email{{\tt mellet@math.umd.edu}}

\address{Beno\^{\i}t Perthame\hfill\break\indent
Sorbonne Universit\'es, UPMC Univ. Paris 06, CNRS, UMR 7598, Laboratoire Jacques-Louis Lions, INRIA \'Equipe MAMBA, 4, place Jussieu 75005, Paris, France.} \email{{\tt
benoit.perthame@upmc.fr}}

\address{Fernando Quir\'{o}s\hfill\break\indent
Departamento  de Matem\'{a}ticas, Universidad Aut\'{o}noma de Madrid
\hfill\break\indent 28049-Madrid, Spain.} \email{{\tt
fernando.quiros@uam.es} }

\begin{document}

\begin{abstract}
We consider weak solutions to a problem modeling tumor growth. Under certain conditions on the initial data, solutions can be obtained  by passing to the stiff (incompressible) limit in a porous medium type problem with a Lotka-Volterra source term describing the evolution of the number density of cancerous cells.
We prove that such limit solutions solve a free boundary problem of Hele-Shaw type. We also obtain regularity properties, both for the solution and for its free boundary.

The main new difficulty arises from the competition between the growth due to the source, which keeps the initial singularities, and the free boundary which invades the domain with a regularizing effect. New islands can be generated at singular times.
\end{abstract}

\maketitle

\noindent{\makebox[1in]\hrulefill}\newline
\section{Introduction}
\setcounter{equation}{0}
We consider solutions $(n,p)$ to the nonlinear parabolic equation
\begin{equation}
\label{eq:main}
\pa_t n = \Delta p + n G(p) \quad\text{in }\mathcal{D}'(\mathbb{R}^N\times\mathbb{R}_+),
\end{equation}
where the pair $(n,p)$ lies in the so called Hele-Shaw graph,
\begin{equation}
\label{eq:Hele-Shaw.graph}
0\le n\le1,\quad p\geq 0, \quad \text{and }  p = 0 \quad \text{for }  \; 0\leq n< 1,
\end{equation}
with an initial data
\begin{equation}
\label{eq:initial.data}
n(\cdot,0)=n^0\in L^1_+(\mathbb{R}^N).
\end{equation}
This problem was studied in~\cite{PQV} as a simple mechanical model for the propagation of tumors, following more elaborate models in \cite{bresch,Lowengrub_survey}. In this setting $n$ stands for the density of  cancerous cells, and $p$ for the pressure.
The function $G$ in the source term, taking account of pressure limited cell multiplication by division,  satisfies
\begin{equation}\label{eq:G}
G\in C^1([0,\infty)), \quad G'(\cdot)<0, \quad
\mbox{there exists }  p_M>0 \mbox{ such that } G(p_M)=0.
\end{equation}
The threshold pressure $p_M$, sometimes called homeostatic pressure, is the smallest pressure that prevents cell multiplication because of contact inhibition. For simplicity, and without loss of generality, the maximum packing density of cells has been set to $n=1$; see~\eqref{eq:Hele-Shaw.graph}.

\medskip

\noindent\textit{Some notations. }  We denote
$Q=\mathbb{R}^N\times\mathbb{R}_+$, and, for $T>0$, $Q_T=\mathbb{R}^N\times(0,T)$.  Given $g:Q\to\mathbb{R}$, we will use several times the
abridged notation $g(t)$  to describe the function $x\to g(x,t)$.

\medskip

Whenever the initial data satisfy
\begin{equation}
\label{eq:condition.initial.data}
n^0\in BV(\mathbb{R}^N),\qquad 0\le n^0\le 1\quad\text{a.e.~in }\mathbb{R}^N,
\end{equation}
existence of a solution to~\eqref{eq:main}--\eqref{eq:initial.data} can  be proved by passing to the \emph{stiff} (incompressible) limit $\gamma\to\infty$ for weak solutions $(n_\gamma,p_\gamma)$ to the porous medium type problem
\begin{equation}\label{eq:pme}
\pa_t n_\gamma - \div(n_\gamma \na p_\gamma) = n_\gamma G(p_\gamma)\quad \text{in }Q, \qquad n_\gamma(0)=n_\gamma^0\quad\text{in }\mathbb{R}^N,
\end{equation}
where the density and the pressure are related by the law of state
\begin{equation}
\label{eq:law.of.estate}
p_\gamma=P_\gamma(n_\gamma),\qquad P_\gamma(n)=n^\gamma,\quad \gamma>1,
\end{equation}
and the approximate initial data satisfy
\begin{equation}
\label{eq:approximation.initial.data}
n_\gamma^0\geq 0 , \quad  P_\gamma(n^0_\gamma)\leq p_M, \quad  n_\gamma^0\in BV(\mathbb{R}^N),\qquad
\lim_{\gamma\to\infty}n_\gamma^0=n^0\quad\text{in }BV(\mathbb{R}^N).
\end{equation}
Indeed, as $\gamma\to \infty$, both $n_\gamma$ and $p_\gamma$ converge strongly in $L^1(Q_T)$ (for all $T>0$) respectively to functions $n\in C\big([0,\infty);L^1(\mathbb{R}^N)\big)$ and $p$ satisfying~\eqref{eq:main}--\eqref{eq:initial.data}, and
\begin{equation}
n, \; p\in BV(Q_T), \qquad p\le p_M, \qquad \int_{Q_T} |\nabla p|^2 \leq C(T);
\end{equation}
see~\cite{PQV}.
In what follows, we will denote solutions obtained through this approximation procedure as \emph{stiff limit solutions}.

\medskip

\noindent\emph{Remarks. } (a) If~\eqref{eq:condition.initial.data} holds, an approximating family $\{n_\gamma^0\}$ with the required properties always exist: take for example $n^0_\gamma=p_M^{1/\gamma}n^0$.

\noindent (b) If,  in addition to~\eqref{eq:condition.initial.data}, the initial data are compactly supported, there are approximating families such that $\mathop{\rm supp\,}n^0_\gamma\subset B_R(0)$ for some $R>0$ and all $\gamma>1$. Stiff limit solutions corresponding to these special families are compactly supported in space for all times, and lie within a natural class in which there is uniqueness; see~\cite{PQV} for the details. Hence, they are all the same, no matter the approximating family.

\noindent (c) The class of initial data under consideration contains the subclass of indicator functions of sets of finite perimeter (Cacciopoli sets), which is specially important for the application that we have in mind.

\medskip

The main aim of this paper is to prove that stiff limit  solutions to~\eqref{eq:main}--\eqref{eq:initial.data} are solutions to a Hele-Shaw type free boundary problem, as suggested in~\cite{PQV}.

We first notice that we can rewrite \eqref{eq:pme} as
\begin{equation}\label{eq:pmep}
\pa_t p_\gamma =\gamma p_\gamma\big(\Delta p_\gamma + G(p_\gamma)\big) + |\na p_\gamma|^2 .
\end{equation}
Hence, in the limit we expect to have the complementarity relation
\begin{equation}
\label{eq:complementarity.relation}
p(\Delta p+G(p))=0.
\end{equation}
In fact, in~\cite{PQV} this was proved to be true in a certain $H^1$ sense. On the other hand, where $p=0$ we should have $\partial_t n=nG(0)$, and hence an exponential growth for $n$; see~\eqref{eq:main}. Our first result, proved in Section~\ref{sect:the.set.Omega(t)}, shows that these expected results indeed hold true.

\begin{theorem}[Structure of the solution]
\label{thm:structure.solution}
Let $(n,p)$ be a stiff limit solution to~\eqref{eq:main}--\eqref{eq:initial.data}.
Then, for all $t>0$, there exists a set $\Omega(t)$ such that
$$
\begin{array}{l}
p(t)>0, \quad n(t)=1 \quad \text{a.e.~in }\Omega(t),
\\[8pt]
\Delta p(t) + G(p(t)) = 0\quad  \mbox{in } \mathcal D'(\mathrm{Int}(\Omega(t))),
\\[8pt]
p(t)=0,\quad n(t)=\text{\rm e}^{G(0)t} n^0 \quad \text{a.e.~in }\mathbb{R}^N\setminus \Omega(t).
\end{array}
$$
\label{thm:main1}
\end{theorem}
Here $\mathrm{Int}(\Omega(t))$ denotes the interior of the set $\Omega(t)$ (we will see later that $\Omega(t)$ is open for almost every $t>0$). Note that the variational form of the elliptic equation for $p$ is not available at this stage because we do not know that $p \in H^1_0\big(\mathrm{Int}(\Omega(t))\big)$ by lack of regularity of $\Omega(t)$.

We next focus our attention on the regularity of the \emph{free boundary}, $\partial\Omega(t)$.  The main new difficulty arises from the competition between the growth due to the source $G(p)$ when $n<1$, which may generate new islands and also keeps the initial singularities, and the free boundary which invades the domain with a regularizing effect.  Nevertheless, we will prove that if the  region where $p(t)=0$ has  positive Lebesgue density at a point of the free boundary, then it is smooth in a neighbourhood, if a certain non-degeneracy condition is satisfied.  The precise result, obtained in Section~\ref{sect:regularity.w.Omega}, reads as follows
\begin{theorem}[Local regularity of the free boundary]
Assume $N\leq 3$.
Let  $(n,p)$ and $\Omega(t)$ be as in Theorem~\ref{thm:main1}.
Given $x\in\partial\Omega(t)$, if
\begin{equation}
\label{eq:ndg}
n^0\text{\rm e}^{G(0)t} <1 \quad\text{in }B_\delta(x)\quad\text{for some }\delta>0,
\end{equation}
then:
\item[(i)] $\pa\Omega(t)\cap B_\delta(x)$ has finite $(N-1)-$Hausdorff measure;
\item[(ii)] there is a modulus of continuity $\sigma$, and a value $r_0\in(0,\delta)$ such that, if
\begin{equation}
\label{eq:condition.regularity.free.boundary}
\frac{\mathop{\rm MD}\big(\{p(t)=0\}\cap B_r(x)\big)}{r}>\sigma(r)\quad\text{for some }r<r_0,
\end{equation}
then for some $\rho>0$, $\partial\Omega(t)\cap B_\rho(x)$ is a $C^1$-graph. Here \lq\lq\text{\rm MD}'' stands for the minimal diameter\footnote{Given a set $S\in\mathbb{R}^N$, the
minimum diameter of $S$, denoted $\mathop{\rm MD}(S)$, is the infimum among the distances
between pairs of parallel hyperplanes enclosing $S$. This quantity measures the \lq\lq flatness'' of a set.}.
\label{thm:main2}
\end{theorem}

\noindent\emph{Remark. } By virtue of the inequality
$$
\frac{\mathop{\rm MD}\big(\mathcal{A}\cap B_r(x)\big)}{r}\ge c\;\frac{|\mathcal{A}\cap B_r(x)|}{|B_r(x)|},
$$
condition~\eqref{eq:condition.regularity.free.boundary} can be substituted by the following simpler criterion,
\begin{equation}
\label{eq:alternative.condition.regularity.free.boundary}
\lim_{r\to 0}\frac{|\{p(t)=0\}\cap B_r(x)|}{|B_r(x)|}>0.
\end{equation}
Thus, if the free boundary is Lipschitz, it is $C^1$. However, cusp-like singularities are possible.

\medskip

The proofs of Theorems \ref{thm:main1}--\ref{thm:main2} make use of a
well known connection between Hele-Shaw free boundary problems and obstacle problems. Indeed, given a classical solution of the standard Hele-Shaw problem, its time integral (the so-called Baiocchi variable) solves an elliptic obstacle problem.
This led to the notion of variational solutions of the Hele-Shaw problem,  introduced in
\cite{EJ} in the early 80s. This connection has been extended to weak solutions in~\cite{GQ-2001,GQ}, and to viscosity solutions in~\cite{KM}.
Here, we will need to adapt the general approach to account for the source term $G(p)$, but the idea is similar. More precisely, the key tool throughout the paper is the function
\begin{equation*}
\label{eq:Baiocchi.variable}
 w (x,t)=\int_0^t \textrm{e}^{-G(0)s} p(x,s)\, ds,
\end{equation*}
which will be proved in Section~\ref{sect:regularity.w.Omega} to solve an obstacle problem. The set $\widetilde\Omega(t)=\{w(t)>0\}$, and how it relates to $\Omega(t)=\{p(t)>0\}$, both things studied in Section~\ref{sect:the.set.Omega(t)}, will play an important role.

The proof of Theorem \ref{thm:main2} will also require some further regularity for the pressure $p$. This is obtained in Section~\ref{sect:regularity.pressure}, where we prove that $p\in L^2((0,T);W^{1,4}(\mathbb{R}^N))$, which yields the necessary smoothness if $N\le 3$; see formula~\eqref{eq:alpha1HS} below.
Whether the result holds in dimensions $N>3$ is still an open question at this stage, and would require higher integrability estimates for $\na p$.

To complete the characterization of problem~\eqref{eq:main}--\eqref{eq:initial.data} as a Hele-Shaw type problem, we should give a rule for the evolution of the set $ \Omega (t) = \{p(t)>0\}$. Equation~\eqref{eq:pmep} suggests that we should have $\partial_t p=|\nabla p|^2$ at points of the free boundary $\partial\Omega(t)$, or what is the same,  the normal velocity $V$ of the free boundary should be given by Stefan's condition (Darcy's law), $V=|\nabla p|$. In fact, as a simple consequence of Theorem~\ref{thm:structure.solution}, if  $n^0=\chi_{\Omega_0}\in BV(\mathbb{R}^N)$, then  $n(t)=\chi_{\Omega(t)}$ for all $t>0$, and the equation \eqref{eq:main} is a weak statement of the motion of $\Omega(t)$ with the Stefan rule for the free boundary.
However, as we have already seen, the equation for the pressure does not carry all the information of the evolution, and the equation for the velocity of the free boundary has to be modified to take into account \emph{precancerous} regions, where $0<n<1$, if the initial datum is not a characteristic function.  This is the content of the next theorem which is proved in Section~\ref{sect:velocity.free.boundary}.

\begin{theorem} [The Stefan condition]
Under the conditions of  Theorem~\ref{thm:main1}, $p$ satisfies the following equation:
\begin{equation}\label{eq:hsweak0}
 (1- \textrm{e}^{G(0)t} n^0 )  \pa_t \chi_{\Omega(t)}   = \Delta p(t) + G(p(t))\chi_{\Omega(t)} \quad \text{in } \mathcal D'(Q).
\end{equation}
In particular, the normal velocity $V_n(x,t)$ of $\pa\Omega(t)$ at a point $x$ of the free boundary satisfies (in a sense to be made precise later):
\begin{equation}
\label{eq:Stefan.condition}
V_n(x,t) = \frac{|\na p(x,t)|}{1-\text{\rm e}^{G(0)t}n^0(x)},\quad x\in\partial\Omega(t).
\end{equation}
\label{thm:main3}
\end{theorem}

Notice that the formula for the normal velocity at $x$ coincides with the classical Stefan condition when $n^0(x)=0$. Therefore,  \eqref{eq:main}--\eqref{eq:initial.data} is indeed a weak formulation for a Hele-Shaw type free boundary problem if $n^0=\chi_{\Omega_0}\in BV(\mathbb{R}^N)$, as conjectured in~\cite{PQV}.

In a recent paper \cite{KP}, Kim and Pozar study the same problem and prove in particular that
\eqref{eq:Stefan.condition} holds in the viscosity sense. Their approach relies on pointwise arguments and comparison principle methods which are very different from the approach developed in the present paper.
We note that by proceeding as in \cite{KM}, it should also be possible to  show that \eqref{eq:Stefan.condition} holds in the viscosity sense by using the Baiocchi variable~$w$ and the obstacle problem formulation that we introduce in this paper. We do not pursue this here.
Instead, we will show that equation \eqref{eq:hsweak0} implies a weak (distributional) and a measure theoretical interpretation to~\eqref{eq:Stefan.condition} (see Section~\ref{sect:velocity.free.boundary}). We will not prove that these formulations are equivalent to the notion of viscosity solutions developed by
Kim and Pozar. But they are
simpler to derive and well adapted (and sufficient) to our purpose.

We next show, in Section~\ref{sect:eventual.regularity}, that conditions~ \eqref{eq:ndg} and~\eqref{eq:alternative.condition.regularity.free.boundary} are fulfilled at all points of the free boundary $\partial\Omega(t)$ for all $t$ large enough, thus proving that the free boundary eventually becomes smooth:
\begin{theorem}[Eventual regularity of the free boundary]
Under the assumptions of~Theorem~\ref{thm:main2}, if in addition $n^0_\gamma \subset B_R(0)$ for some $R>0$, and $n^0\ge \alpha\chi_{B_r(\bar x)}$ for some $\alpha\in(0,1]$, $r>0$, and $\bar x\in\mathbb{R}^N$, then there is a time $T_0$ such that the free boundary $\partial\Omega(t)$ is  starshaped and smooth for all $t\ge T_0$.  Consequently, $p$ is Lipschitz continuous in space and time for $t>T_0$.
\label{thm:main4}
\end{theorem}

Finally, as a complement of the convergence results in~\cite{PQV}, we study the convergence of the positivity sets $\{n_\gamma(t)>0\}$ as $\gamma\to\infty$. This analysis is performed in Section~\ref{convergence.positivity.sets}.
\begin{theorem}[Convergence of the positivity set]
Assume the hypotheses of Theorem~\ref{thm:main1}. If, in addition,  $n^0$ has compact support, and $\{n^0_\gamma>0\}\to\{n^0>0\}$ in the sense of Hausdorff distance,  then, for all $t>0$,
$\{n_\gamma(t)>0\}\to\{n(t)>0\}$ in the sense of Hausdorff distance.
\label{thm:convergence.positivity.sets}
\end{theorem}


\section{The sets $\Omega(t)$ and $\widetilde \Omega(t)$, and the structure of the solution}
\label{sect:the.set.Omega(t)}
\setcounter{equation}{0}

This section is devoted to prove Theorem~\ref{thm:main1}. This will require some preliminary work, namely: (i) giving a suitable pointwise definition for $p$; (ii) obtaining some regularity for the time integral $w$ defined in~\eqref{eq:Baiocchi.variable}; and (iii) relating the positivity sets of $w(t)$ and $p(t)$.

\subsection{Pointwise definition of $p$}
We first recall some properties of stiff limit solutions $(n,p)$  to~\eqref{eq:main}--\eqref{eq:initial.data}, established in \cite{PQV},  that will be needed in this section, namely
\begin{eqnarray}
\label{eq:L1n}
&&\int_{\mathbb{R}^N} n(t)\leq \textrm{e}^{G(0)t} \int_{\mathbb{R}^N} n^0\quad\text{for all }t\ge0,
\\[8pt]
\label{eq:L1p}
&&\int_{\mathbb{R}^N} p(t)\leq p_M \textrm{e}^{G(0)t}  \int_{\mathbb{R}^N} n^0\quad\text{for a.e. } t>0,
\\[8pt]
\label{eq:delta}
&&\Delta p + G(p)  \geq 0\quad\text{in }\mathcal D'(Q),
\\[8pt]
\label{eq:patp}
&&\pa_t p\geq 0\quad\text{in }\mathcal D'(Q).
\end{eqnarray}
The first two inequalities are quite straightforward. Indeed, integrating \eqref{eq:pme} with respect to $x$ and using Gronwall's inequality and the monotonicity of $G$, we immediately get
\begin{equation*}\label{eq:L1ng}
\int_{\mathbb{R}^N} n_\gamma(t)\leq \textrm{e}^{G(0)t} \int_{\mathbb{R}^N} n^0_\gamma \quad\text{for all }\gamma >1,
\end{equation*}
which, combined with the $L^\infty$ bound $p_\gamma\leq p_M$, yields
\begin{equation}\label{eq:L1pg}
\int_{\mathbb{R}^N} p_\gamma(t) \leq \textrm{e}^{G(0)t} p_M^{\frac{\gamma-1}{\gamma}} \int_{\mathbb{R}^N} n^0_\gamma\leq C \textrm{e}^{G(0)t}\quad\text{for all }\gamma >1.
\end{equation}

The \lq semiconvexity' inequality~\eqref{eq:delta} and the monotonicity estimate~\eqref{eq:patp} follow from the next crucial lemma, which is a variant of some very classical computation for the porous media equation (without reaction term) due to Aronson and Bénilan \cite{AB}; see  \cite{PQV} for the proof.
\begin{lemma}\label{lem:delta}
Let $c=\min _{p\in[0,P_M]} (G(p)-pG'(p))>0$. For all $\gamma>1$,
\begin{eqnarray}
\label{eq:deltag}
&\displaystyle\Delta p_\gamma(t) + G(p_\gamma(t))  \geq - c \frac{\text{\rm e}^{-\gamma ct}}{1-\text{\rm e}^{-\gamma c t}}\quad\text{in }\mathcal{D}'(Q),
\\[8pt]
\label{eq:patpg}
&\displaystyle\pa_t p_\gamma(t)\geq -\gamma p_\gamma(t)c \frac{\text{\rm e}^{-\gamma ct}}{1-\text{\rm e}^{-\gamma c t}}\quad\text{in }\mathcal{D}'(Q).
\end{eqnarray}
\end{lemma}

We note that $n(t)$ is well defined, and belongs to $L^1(\RR^N)$, for \emph{all} $t$ , while the function  $p(t)$ is known to belong to $L^1(\RR^N)$ only for almost every $t$.
However, in what follows, it will be convenient to be able to talk about the pointwise value of the pressure.
We thus start with the following lemma.
\begin{lemma}[Pointwise definition of the pressure]\label{lem:pointwise}
Let $(n,p)$ be a stiff limit solution to~\eqref{eq:main}--\eqref{eq:initial.data}. The pressure $p$ can be redefined in a set of measure $0$ so that
\begin{equation}
\label{eq:pointwise.definition}
p(x_0,t_0)=\lim_{r\to 0}  \frac{1}{|B_r|}\int_{B_r(x_0)}\frac{1}{r^2} \int_{t_0}^{t_0+r^2} p(x,t)\,  dt\, dx\quad\text{for all }(x_0,t_0)\in \overline{Q}.
\end{equation}
With this pointwise definition:
\item[(i)] $p$ is upper-semicontinuous, $\limsup_{(x,t)\to(x_0,t_0)} p(x,t) \leq  p(x_0,t_0)$;
\item[(ii)] for all $x$, the function $t\mapsto p(x,t)$ is monotone non-decreasing and continuous to the right,
\begin{equation} \label{eq:ppoint}
\lim_{t\to t_0^+} p(x,t) = p(x,t_0)\qquad\mbox{for all $x\in \mathbb{R}^N$.}
\end{equation}
\end{lemma}

Inequalities \eqref{eq:delta}--\eqref{eq:patp} imply that the pressure
satisfies
$$
\pa_t p+ \Delta p \geq -G(0)\quad\text{in }\mathcal{D}'(Q),
$$
and so the function
$$
(x,t)\mapsto p(x,-t)+\frac{G(0)}{2N} |x|^2
$$
is a sub-solution of the heat equation.
We could thus use some version of the mean-value formula for sub-solutions of the heat equation to prove Lemma~\ref{lem:pointwise}. However, this mean-value formula involves heat balls. To circumvent them,  we will use inequalities \eqref{eq:delta} and \eqref{eq:patp} separately.

\begin{proof}[Proof of Lemma \ref{lem:pointwise}]
Inequality \eqref{eq:delta} gives
$$
\Delta p \geq -G(0) \quad \mbox{in } \mathcal D'(Q)
$$
and so, for all $t_0\geq 0$ and $s>0$, we have
$$
\Delta\left( \frac{1}{s^2}\int_{t_0}^{t_0+s^2} p(t)\,dt\right) \geq -G(0)\quad\text{in }\mathcal D'(\mathbb{R}^N).
$$

The classical Mean Value Property for the Laplace equation thus implies that, for all $x_0\in \mathbb{R}^N$, the function
$$
\phi(r,s)= \frac{1}{|B_r|}\int_{B_r(x_0)}\frac{1}{s^2} \int_{t_0}^{t_0+s^2}\left( p(x,t)  + \frac{G(0)}{2N}(x-x_0)^2\right)\, dt\, dx
$$
is non-decreasing with respect to $r$.
Furthermore, inequality \eqref{eq:patp} implies that the function $s\mapsto \phi(r,s)$ is non-decreasing with respect to $s$ (for $s>0$) for all $r>0$.
We conclude that the function
$$
r\mapsto \phi(r,r) = \frac{1}{|B_r|}\int_{B_r(x_0)}\frac{1}{r^2} \int_{t_0}^{t_0+r^2}\left( p(x,t)  + \frac{G(0)}{2N}(x-x_0)^2\right)\, dt\, dx
$$
is non-decreasing, and so
$ \lim_{r\to 0} \phi(r,r)$
exists for all $(x_0,t_0)$.
By Lebesgue's Differentiation Theorem, this limit is equal to $p(x_0,t_0)$ for almost every $(x_0, t_0)$. Therefore, after redefining $p$ on a set of measure zero, we have~\eqref{eq:pointwise.definition}.

The other properties of $p$ are now easily derived.
Indeed, we also have
$$
p(x_0,t_0)  = \inf_{r>0}  \frac{1}{r^2 |B_r|} \int_{B_r(x_0)} \int_{t_0}^{t_0+r^2}\left( p(x,t) + \frac{G(0)}{2N}(x-x_0)^2\right)\, dt\, dx
$$
and so $p$ is upper-semicontinuous (as the infimum of a family of continuous functions).

Finally,
it is readily seen that
$$
p(x,t)\leq p(x,s) \quad \text{for all }x\in \mathbb{R}^N,\text{ and }0<t<s.
$$
Therefore, the limit $\lim_{t\to t_0^+} p(x,t)$ exists, and satisfies $p(x,t_0)\le\lim_{t\to t_0^+} p(x,t)$, which combined with the upper semicontinuity yields~\eqref{eq:ppoint}.
\end{proof}

The one-sided continuity of $p$ with respect to $t$ gives us more information on the complementarity relation~\eqref{eq:complementarity.relation}.
\begin{corollary}\label{cor:pdelta}
For all $t>0$, the function $p(t)$ satisfies
\begin{eqnarray}\label{eq:pdelta1}
&\Delta p(t)+G(p(t)) \geq 0 \quad\text{in } \mathcal D'(\RR^N),
\\[8pt]
\label{eq:pdelta2}
&\Delta p(t)+G(p(t)) = 0 \quad \text{in } \mathcal D'(\mathrm{Int}(\{p(t)>0 \})).
\end{eqnarray}
\end{corollary}

\noindent\emph{Remark. } The set $\{ p(t)>0 \}$ may not be open, though we will show later on that it is open for almost every $t>0$.
\begin{proof}
Let $\vphi\in\mathcal D(\RR^N)$, $\varphi\ge0$. We define the function $H\in L^\infty(\mathbb{R}_+)$
by
$$
H(t)=\int_{\RR^N} \big(p(t) \Delta \vphi + G(p(t))\vphi\big).
$$
Then \eqref{eq:delta} implies that $H\ge0$
in $\mathcal D'(\mathbb{R}_+)$, and hence
for almost every $t>0$. Finally, the one-sided continuity \eqref{eq:ppoint} implies that $H(t)\ge0$ for all $t>0$,
which gives \eqref{eq:pdelta1}.

Next, we fix $t_0>0$ and we note that the monotonicity of $p$ with respect to $t$ implies that $p(t)>0$ in $\{p(t_0)>0 \}$ for all $t>t_0$.
In particular, we deduce that $n(x,t)=1$ for almost every $(x,t)\in \Sigma_{t_0}:=\{p(t_0)>0 \}\times(t_0,\infty)$, and so equation \eqref{eq:main} yields
$$
\Delta p+G(p) = 0  \quad \text{in }    \mathcal D'(\mathrm{Int}(\Sigma_{t_0})).
$$

We can then proceed as in the first part of the proof to show that for all $t\geq t_0$ (note that we can take $t=t_0$ because of the right hand side continuity) we have
$$
\Delta p(t_0)+G(p(t_0)) = 0  \quad \text{in }  \mathcal D'(\mathrm{Int}(\{p(t_0)>0 \})),
$$
which gives \eqref{eq:pdelta2}.
\end{proof}



\subsection{The function $ w $}
As mentioned in the Introduction,
it is a classical fact that the time integral of the solution of the Hele-Shaw free boundary problem has better regularity (in time and in space) than the solution itself.
Following that idea, but taking into account the presence of a source term in our equation,
 we introduce the function
\begin{equation} \label{eq:bw}
 w (x,t)=\int_0^t \textrm{e}^{-G(0)s} p(x,s)\, ds.
\end{equation}
Note that $ w $ is clearly Lipschitz in time with values in $L^1(\mathbb{R}^N)\cap L^\infty(\mathbb{R}^N)$.
Our first step is to prove that this function is in fact continuous with respect to $x$.
This will follow from the next simple lemma.
\begin{lemma}
For all $t>0$, the function $w(t)$ defined by \eqref{eq:bw} satisfies
\begin{equation}\label{eq:Deltabw}
\Delta w(t) = \underbrace{\text{\rm e}^{-G(0)t} n(t)-n^0 + \int_0^t \text{\rm e}^{-G(0)s} n(s)\big(G(0)-G(p(s))\big)\, ds}_{ =: \tilde F(t)} \quad\text{a.e.~in }\mathbb{R}^N.
\end{equation}
\end{lemma}
\begin{proof}
Equation \eqref{eq:Deltabw} in~$\mathcal{D}'(\mathbb{R}^N)$ is obtained by first rewriting \eqref{eq:main} as
$$
\pa_t (\textrm{e}^{-G(0)t} n(t)) = \Delta(\textrm{e}^{-G(0)t} p(t)) - n \textrm{e}^{-G(0)t} \big(G(0)-G(p(t))\big),
$$
and then integrating with respect to $t$.
 On the other hand, $n\in C([0,\infty);L^1(\mathbb{R}^N))$, so the right hand side in \eqref{eq:Deltabw}  belongs to $L^1(\mathbb{R}^N)$, and hence the equation is satisfied almost everywhere.
\end{proof}

\begin{corollary}\label{cor:preg}
For all $t>0$, $w(t)\in W^{2,p}(\mathbb{R}^N)$ for all $p\in(1,\infty)$.
In particular,  $w(t)\in C^{1,\alpha}(\mathbb{R}^N)$ for all $\alpha\in(0,1)$.
\end{corollary}
\begin{proof}
Let $\tilde F$ be as in~\eqref{eq:Deltabw}. Using that $n\leq 1$, $n^0\geq 0$ and $G(p)\geq 0$, since $p\le p_M$, we get
$$
\tilde F(t) \leq \textrm{e}^{-G(0)t}  + \int_0^t \textrm{e}^{-G(0)s} G(0)\, ds = 1.
$$
Furthermore, since $n\geq 0$ and $G(0)\ge G(p)$, we also have
$$
\tilde F\geq -n^0 \geq -1.
$$
We deduce that $\|\tilde F(t)\|_{L^\infty(\mathbb{R}^N)}\leq 1$.
Note also that (using \eqref{eq:L1n})
$$
\| \tilde F(t)\|_{L^1(\mathbb{R}^N)} \leq (2+t G(0)) \|n^0\|_{L^1(\mathbb{R}^N)}.
$$
We thus have
$$
\|\tilde F(t)\|_{L^p(\mathbb{R}^N)}\leq C(t) \|n^0\|_{L^1(\mathbb{R}^N)} ^{1/p}\quad\text{for all }p\in (1,\infty),
$$
and  Calder\'on-Zygmund estimates give
$$
\| w(t)\|_{W^{2,p}(\mathbb{R}^N)} \leq C(t) \|n^0\|_{L^1(\mathbb{R}^N)} ^{1/p} \quad\text{for all } p\in(1,\infty).
$$
Finally, Sobolev's embeddings  imply that  $w(t)$ is a $C^{1,\alpha}$ function for all $\alpha\in(0,1)$.
\end{proof}

We can show also that $w$ is Lipschitz continuous with respect to $(x,t)$.
\begin{corollary}\label{cor:bwcont}
For all $T>0$, the function $ w $ belongs to $W^{1,\infty}_{\rm loc}(Q_T)$.
\end{corollary}
\begin{proof}
For  $p>N$, we have
$ w  \in L^\infty((0,T);W^{2,p}(\mathbb{R}^N)) \subset L^\infty((0,T);W^{1,\infty}(\mathbb{R}^N))$, and we already saw that $w\in W^{1,\infty}((0,T); L^\infty(\RR^N))$.
\end{proof}

\subsection{The positivity sets}

In view of Corollary \ref{cor:preg},
we can introduce the open set
\begin{equation}\label{eq:Omegatilde}
\widetilde \Omega(t)=\{w (t)>0\}.
\end{equation}
Furthermore, Corollary \ref{cor:bwcont} also implies that
$$
\widetilde {\mathcal O} = \{(x,t)\in Q;\,  w (x,t)>0\}
$$
is open.
Since $p$ is a non-negative function,  the function $ w $ defined by \eqref{eq:bw} is clearly non-decreasing with respect to $t$ and so the sets $\widetilde \Omega(t)$ form a non-decreasing family.
\medskip

Next, we introduce the measurable (not necessarily open) set
\begin{equation}\label{eq:Omega(t)}
\Omega (t) = \{p(t)>0\}
\end{equation}
(this set is well defined for all $t>0$ thanks to Lemma \ref{lem:pointwise}).
Since $t\mapsto p(x,t)$ is a non-decreasing function for all $x$,
the sets $\Omega(t)$ form also a non-decreasing family.

\medskip

The proof of Theorem \ref{thm:main1} will mostly follow from the properties of the function $w(t)$ on the sets $\widetilde \Omega(t)$ and $\mathbb{R}^N\setminus \widetilde \Omega(t)$.
Of course,  we would like to say that $\Omega(t)=\widetilde \Omega(t)$; however that may not always be the case because $p$ might not be continuous with respect to $t$.
Nevertheless, we do have the following (simple) result.
\begin{lemma} \label{lem:omegap}
The following inclusions hold:
$$
\widetilde \Omega(t)\subset \Omega (t) \subset\widetilde  \Omega(s)  \quad \text{for all }0<t<s.
$$
\end{lemma}
\begin{proof}
Since $p(x,t)=0$ implies $p(x,s)=0$ for all $s\leq t$ (monotonicity with respect to $t$), it also implies $ w (x,t)=0$, which yields $\widetilde \Omega(t)\subset \Omega (t)$.

Similarly, if $p(x,t_0)>0$, then $p(x,t)\geq p(x,t_0)$ for all $t\geq t_0$, and so
$$
w (x,s) \geq  p(x,t_0)\int_{t_0}^s \textrm{e}^{-G(0)\tau } d\tau >0 \qquad \forall s >t_0,
$$
which gives the second inclusion.
\end{proof}

Note that if $s\mapsto p(x,s)$ is continuous at $s=t$ for all $x\in \mathbb{R}^N$, then $p(x,t)>0$ implies $w(x,t)>0$ (since $p(x,s)>0$ for $s\in(t-\delta,t]$), and so  $\widetilde \Omega(t) = \Omega (t) $ in such a situation.
While this might not be true for all $t>0$, we will see in Proposition \ref{prop:n} that it is indeed the case for almost every $t>0$.

\medskip

\subsection{Proof of Theorem \ref{thm:main1}}
Theorem \ref{thm:main1} will now follow easily from the next proposition, the proof of which
relies mostly on \eqref{eq:Deltabw}.
\begin{proposition}\label{prop:n}
For all $t>0$,
\begin{equation}\label{eq:defn}
n(t) = \left\{
\begin{array}{ll}
1 & \text{a.e.~in }\widetilde \Omega(t), \\
\text{\rm e}^{G(0)t} n^0\quad &\text{a.e.~in }\mathbb{R}^N\setminus \widetilde \Omega(t).
\end{array}
\right.
\end{equation}
Moreover,
\begin{equation}\label{eq:Omegabw}
\Omega (t) \setminus \widetilde \Omega(t) \subset \{x\in \mathbb{R}^N\,;\,  n^0(x) =\text{\rm e}^{-G(0)t}\},
\end{equation}
and so \eqref{eq:defn} also holds with $\Omega (t)$ instead of $\widetilde\Omega(t)$.
\end{proposition}

\begin{proof}
Lemma \ref{lem:omegap} gives that $p(t)>0$ in $\widetilde\Omega(t)$, and so the condition \eqref{eq:Hele-Shaw.graph} gives $n=1$ a.e.~in $\widetilde\Omega(t)$.

Next, the fact that $ w(t)  \in W^{2,p}(\mathbb{R}^N)$ implies that $\Delta  w (t)=0$ a.e.~in $\{ w  (t)=0\}=\mathbb{R}^N\setminus\widetilde\Omega(t)$, and so \eqref{eq:Deltabw} gives
$$
\textrm{e}^{-G(0)t} n(t)-n^0 + \int_0^t \textrm{e}^{-G(0)s} n(s)\big(G(0)-G(p(s))\big)\, ds = 0 \quad \text{a.e.~in }  \mathbb{R}^N\setminus \widetilde\Omega(t).
$$
Furthermore, Lemma \ref{lem:omegap} also implies that
for all $s<t$, we have
$ p(s)=0$ a.e.~in $\mathbb{R}^N\setminus\widetilde\Omega(t)$ and so the above formula reduces to
$$\textrm{e}^{-G(0)t} n(x,t)-n^0(x)  = 0 \quad \text{a.e.~in }\mathbb{R}^N\setminus\widetilde\Omega(t),
$$
which completes the proof of~\eqref{eq:defn}.

Finally, we note that the proof above actually gives $n(t)=1$ a.e.~in $\Omega(t)$, and not just in $\widetilde\Omega(t)$.
So, for almost all $x$ in $\Omega(t)\setminus\widetilde\Omega(t)$ we have $n(x,t)=\textrm{e}^{G(0)t} n^0(x) =1 $, which implies \eqref{eq:Omegabw}.
\end{proof}

\noindent\emph{Remark. } Once $\widetilde\Omega(t)\supset\mathop{\rm supp} n^0$, we have $\Omega(t)=\widetilde\Omega(t)$.

\medskip

We recall  that the function $n(t)$ was well defined for all $t$ as a function in $L^1(\RR^N)$.
In view of the proposition above, we can now assume that $n$ is defined pointwise by the formula
\begin{equation}\label{eq:defnb}
n(x,t) = \left\{
\begin{array}{ll}
1 & \mbox{ for all $x$ in } \widetilde\Omega(t), \\
\textrm{e}^{G(0)t} n^0(x) &  \mbox{  for all $x$ in }  \mathbb{R}^N\setminus \widetilde \Omega(t).
\end{array}
\right.
\end{equation}

Formula \eqref{eq:Omegabw} can also be used to characterize the times for which $\Omega(t)\neq \widetilde\Omega(t)$.
Indeed, since $n^0\in BV(\mathbb{R}^N)$, a classical consequence of the coarea formula is that $H^{n-1}(\{n^0=\lambda\})<\infty$ for almost all $\lambda\in \mathbb{R}$. In particular, the Lebesgue measure of
$ \{x\in \mathbb{R}^N\,;\,  n^0(x) =\textrm{e}^{-G(0)t}\}$ vanishes for almost all $t>0$.
\begin{corollary}\label{cor:tnt}
Let $(n,p)$ be a stiff limit solution to~\eqref{eq:main}--\eqref{eq:initial.data}. Then
$$
\widetilde \Omega(t)=\Omega (t) \quad\text{a.e. } t\geq 0.
$$
\end{corollary}

It is also now immediate to show that
if $n^0$ is the characteristic function of a Cacciopoli set, $n^0=\chi_{\Omega_0}$, then
$\Omega(t)=\widetilde \Omega(t)$ for all $t>0$,  and $n(t)$ is also the characteristic function of a set for all $t>0$, $n(t)=\chi_{\Omega(t)}$.

Since we know that $n(t)\in BV(\mathbb{R}^N)$, we deduce that $\widetilde\Omega(t)$ is a set with finite perimeter. However,  such a result does not even imply that the boundary $\pa\widetilde\Omega(t)$ has zero Lebesgue measure; see \cite[pp.~7--8]{Giusti}. In  order to get further information about $\pa\widetilde\Omega(t)$, we need further regularity estimates on the pressure $p$.

\section{Further regularity of the pressure $p$}
\label{sect:regularity.pressure}
\setcounter{equation}{0}

In this section, we prove some new estimates for $p$, compared to those in \cite{PQV}, that will be used in the next section to study the regularity of $\pa\widetilde\Omega(t)$.

We now give the main result of this section.
\begin{proposition}\label{prop:alpha-HS} Let $(n,p)$ be a stiff limit solution to~\eqref{eq:main}--\eqref{eq:initial.data}. There exists a constant such that the stiff limit pressure satisfies, for $a. \; e. \; t>0$
\begin{equation}\label{eq:H1limit}
\int_{\mathbb{R}^N} |\na p(t)|^2\leq C\left( 1 + \frac{1}{\gamma t}\right) \text{\rm e}^{G(0)t}.
\end{equation}

For all $\alpha \geq 0$ and $T>0$, there exists a locally bounded constant $C(T,\alpha)$ such that
\begin{eqnarray}\label{eq:alpha1HS}
&\displaystyle{\alpha} \int_0^T \int_{\mathbb{R}^N} p^{\alpha-1} |\na p|^4 \leq C(T,\alpha),
\\
\label{eq:alpha2HS}
&\displaystyle\alpha \int_0^T \int_{\mathbb{R}^N} | \Delta p^{(3  + \alpha)/2}|^2 \leq C(T,\alpha).
\end{eqnarray}
\end{proposition}

This proposition is a direct consequence of the following  estimates which use strongly the results of Lemma~\ref{lem:delta}.
\begin{lemma}\label{lm:H1}
There exists a constant $C$ such that
\begin{equation}\label{eq:H1}
\int_{\mathbb{R}^N} |\na p_\gamma(t)|^2\leq C\left( 1 + \frac{1}{\gamma t}\right) \text{\rm e}^{G(0)t}.
\end{equation}
For all $T>0$,  there is a locally bounded constant $C(T)$ such that
\begin{equation}\label{eq:H1gamma}
\gamma\int_\frac{1}{\gamma}^T\int p_\gamma\big(\Delta p_\gamma +G(p_\gamma)\big)^2 \leq C(T).
\end{equation}
For all $\alpha \geq 0$ and $T>0$, there exists a locally bounded constant $C(T,\alpha)$ such that
\begin{equation}\label{eq:alpha1}
{\alpha} \int_\frac{1}{\gamma}^T \int_{\mathbb{R}^N} p_\gamma^{\alpha-1} |\na p_\gamma|^4 \leq C(T,\alpha).
\end{equation}
\end{lemma}

Because of Lipschitz continuity at the free boundary, the inequalities~\eqref{eq:alpha1HS} and \eqref{eq:alpha1} are sharp in terms of $\alpha$ and we cannot expect a bound on $\int_0^T\int{\mathbb{R}^N} \frac{|\nabla p|^4}{p}$.

\noindent {\em Proof of Proposition \ref{prop:alpha-HS}}.
The estimate \eqref{eq:H1limit} and \eqref{eq:alpha1HS} follow immediately from  \eqref{eq:H1} and \eqref{eq:alpha1}. To prove \eqref{eq:alpha2HS}, we take $\alpha >0$ and multiply  equation \eqref{eq:H1gamma} by $p^\alpha$. Then, in $L^2( Q_T)$, we have
$$
 p_\gamma^{(1+\alpha)/2}  \big(\Delta p_\gamma+G(p_\gamma)\big) = {\rm div}\big( p_\gamma^{(1+\alpha)/2}  \nabla p_\gamma\big) - \frac{1+\alpha}{2}p_\gamma^{(\alpha-1)/2} |\nabla p_\gamma|^2 +p_\gamma^{(1+\alpha)/2} G(p_\gamma) \underset{\gamma \to \infty}{\longrightarrow}  0  .
$$
Taking into account \eqref{eq:alpha1}, we obtain \eqref{eq:alpha2HS}.
\qed

These bounds also provide us with a convenient way to define the strong form of the complementary relation, which reads formally, for $\alpha >0$,
$$
 p^{(1+\alpha)/2}  \big(\Delta p +G(p)\big) =0.
$$
It is rigorously stated as
\begin{equation}\label{eq:complementary_strong}
{\rm div}\big( p^{(1+\alpha)/2}  \nabla p \big) - \frac{1+\alpha}{2}p^{(\alpha-1)/2} |\nabla p |^2 +p^{(1+\alpha)/2} G(p ) =0 \quad \text{a.e.~in }Q.
\end{equation}

\noindent {\em Proof of Lemma \ref{lm:H1}}.
To  prove \eqref{eq:H1}, we integrate \eqref{eq:pmep} with respect to $x\in \mathbb{R}^N$ and use \eqref{eq:patpg}. We get:
\begin{align*}
(\gamma-1)\int_{\mathbb{R}^N} |\na p_\gamma(t)|^2 & = \gamma \int_{\mathbb{R}^N} p_\gamma(t) G(p_\gamma(t)) - \int_{\mathbb{R}^N}\pa_t p_\gamma(t) \\
&\leq   \gamma \int_{\mathbb{R}^N} p_\gamma(t) G(p_\gamma(t)) + \gamma c \frac{\textrm{e}^{-\gamma ct}}{1-\textrm{e}^{-\gamma c t}} \int_{\mathbb{R}^N}  p_\gamma(t).
\end{align*}
We deduce that
\begin{align*}
\int_{\mathbb{R}^N} |\na p_\gamma(t)|^2  & \leq \frac{\gamma}{\gamma-1}  \int_{\mathbb{R}^N} p_\gamma(t) G(p_\gamma(t))+\frac{ \gamma}{\gamma-1} c \frac{\textrm{e}^{-\gamma ct}}{1-\textrm{e}^{-\gamma c t}}\int_{\mathbb{R}^N}  p_\gamma(t) \\
& \leq \frac{\gamma}{\gamma-1} \left(1+\frac{\textrm{e}^{-\gamma ct}}{1-\textrm{e}^{-\gamma c t}}\right)  \int_{\mathbb{R}^N}  p_\gamma(t),
\end{align*}
which gives \eqref{eq:H1} after using \eqref{eq:L1pg}.

To prove \eqref{eq:H1gamma}, we define $H(p)= \int_0^p G(q)dq$, multiply \eqref{eq:pmep} by $\Delta p_\gamma +G(p_\gamma)$, and integrate,
\begin{align*}
&\frac{d}{dt} \int_{\mathbb{R}^N} \left(-\frac{|\na p_\gamma(t)|^2 }{2} + H(p_\gamma(t))\right)\\
 &\qquad\qquad= \gamma\int_{\mathbb{R}^N} p_\gamma\big(\Delta p_\gamma(t) +G(p_\gamma(t))\big)^2 +\int |\na p_\gamma(t)|^2 \big(\Delta p_\gamma(t)+G(p_\gamma(t))\big)\\
&\qquad\qquad\geq
  \gamma\int_{\mathbb{R}^N} p_\gamma(t)\big(\Delta p_\gamma(t) +G(p_\gamma(t))\big)^2 - c \frac{\textrm{e}^{-\gamma ct}}{1-\textrm{e}^{-\gamma c t}} \int_{\mathbb{R}^N} |\na p_\gamma(t)|^2.
\end{align*}
Therefore,
\begin{align*}
\gamma \int_\frac{1}{\gamma}^T\int_{\mathbb{R}^N} p_\gamma\big(\Delta p_\gamma +&G(p_\gamma)\big)^2 \\
&\leq  c \int_\frac{1}{\gamma}^T \left(\frac{\textrm{e}^{-\gamma ct}}{1-\textrm{e}^{-\gamma c t}} \int_{\mathbb{R}^N}|\na p_\gamma(t)|^2\right) \, dt  + \int_{\mathbb{R}^N}\left( \frac{|\na p_\gamma (\frac{1}{\gamma})|^2}{2} + H(p_\gamma(T))\right).
\end{align*}
 The result follows using \eqref{eq:H1}.

It remains to prove~\eqref{eq:alpha1}. We multiply \eqref{eq:pmep} by $p_\gamma^{\alpha} \big(\Delta p_\gamma +G(p_\gamma)\big)$ and integrate in space,
\begin{align*}
\int_{\mathbb{R}^N}  p_\gamma^{\alpha}(t) \pa_t p_\gamma(t)  \big(\Delta p_\gamma(t) +G(p_\gamma(t))\big) =&\int_{\mathbb{R}^N}  \gamma p_\gamma^{1+\alpha}(t) \big(\Delta p_\gamma(t)+G(p_\gamma(t))\big)^2\\&
+ \int_{\mathbb{R}^N} |\na p_\gamma(t)|^2 p_\gamma^{\alpha}(t) \big(\Delta p_\gamma(t) +G(p_\gamma(t))\big).
\end{align*}
To handle the first term in the left hand side, we integrate by parts and write
\begin{align*}
\int_{\mathbb{R}^N}  p_\gamma^{\alpha}(t) \pa_t p_\gamma(t)  \Delta p_\gamma(t) & = -\frac{1}{2}\int_{\mathbb{R}^N} p_\gamma^{\alpha}(t) \pa_t |\na p_\gamma(t)|^2  -\alpha\int_{\mathbb{R}^N} p_\gamma^{\alpha-1}(t) \pa_t p_\gamma(t) |\na p_\gamma(t)|^2\\
& =  -\frac{1}{2}\frac{d}{dt} \int_{\mathbb{R}^N} p_\gamma^{\alpha}(t)  |\na p_\gamma(t)|^2- \frac{\alpha}{2}  \int_{\mathbb{R}^N} p_\gamma^{\alpha-1}(t)\pa_t p_\gamma(t)  |\na p_\gamma(t)|^2\\
& =  -\frac{1}{2}\frac{d}{dt} \int_{\mathbb{R}^N} p_\gamma^{\alpha}(t)  |\na p_\gamma(t)|^2
-\frac{\alpha}{2}  \int_{\mathbb{R}^N} p_\gamma^{\alpha-1}(t) |\na p_\gamma(t)|^4\\
 &\quad-\frac{\alpha\gamma}{2} \int_{\mathbb{R}^N} p_\gamma^{\alpha}(t)  |\na p_\gamma(t)|^2\big(\Delta p_\gamma(t) +G(p_\gamma(t))\big).
\end{align*}
We deduce, with $H^{(\alpha)}(p)= \int_0^p q^{\alpha} G(q)dq$, that
\begin{equation} \label{eq:energy} \begin{array}{rl}
 \displaystyle -\frac{1}{2}\frac{d}{dt} &\displaystyle \int_{\mathbb{R}^N} p_\gamma^{\alpha}(t)  |\na p_\gamma(t)|^2 + \frac{d}{dt}\int_{\mathbb{R}^N} H^{(\alpha)}(p_\gamma(t))=  \gamma \displaystyle \int_{\mathbb{R}^N}  p_\gamma^{1+\alpha}(t) \big(\Delta p_\gamma(t)+G(p_\gamma(t))\big)^2
 \\[12pt]
  &  + (1+ \frac{\alpha\gamma}{2})\displaystyle \int_{\mathbb{R}^N} p_\gamma^{\alpha}(t)  |\na p_\gamma(t)|^2\big(\Delta p_\gamma(t) +G(p_\gamma(t))\big) +  \frac{\alpha}{2}  \int_{\mathbb{R}^N} p_\gamma^{\alpha-1}(t) |\na p_\gamma(t)|^4.
\end{array}\end{equation}

Using \eqref{eq:deltag}, we conclude that
\begin{align*}
 \frac{\alpha}{2}  &\int_{\mathbb{R}^N} p_\gamma^{\alpha-1}(t) |\na p_\gamma(t)|^4+ \gamma \int_{\mathbb{R}^N}  p_\gamma^{1+\alpha}(t) \big(\Delta p_\gamma(t)+G(p_\gamma(t))\big)^2\\
 & \leq C \left(1+\frac{\alpha}{2} \gamma \right) \frac{\textrm{e}^{-\gamma ct}}{1-\textrm{e}^{-\gamma c t}}
  \int_{\mathbb{R}^N} p_\gamma^{\alpha}(t)  |\na p_\gamma(t)|^2 -\frac{1}{2}\frac{d}{dt} \int_{\mathbb{R}^N} p_\gamma^{\alpha}(t)  |\na p_\gamma(t)|^2 + \frac{d}{dt}\int_{\mathbb{R}^N} H^{(\alpha)}(p_\gamma(t)).
\end{align*}
Using the bound~\eqref{eq:H1limit} in Lemma~\ref{lm:H1},  the result follows after time integration because the second term is controlled thanks to~\eqref{eq:H1gamma}.
\qed

\subsection{Energy disipation}
When $\alpha = \frac 1 \gamma$, equality \eqref{eq:energy} can also be written as
\begin{align*}
\frac{d}{dt} \int_{\mathbb{R}^N} &\left(n_\gamma(t)\, \frac{|\na p_\gamma(t)|^2}{2}  - H^{({\frac 1 \gamma})}(p_\gamma(t))\right)\\
& =-  \frac{1}{2\gamma}  \int_{\mathbb{R}^N} \frac{n_\gamma(t)}{ p_\gamma(t)} \,|\na p_\gamma(t)|^4 - \gamma \int_{\mathbb{R}^N}   n_\gamma(t) p_\gamma(t) \big(\Delta p_\gamma(t)+G(p_\gamma(t))\big)^2 \nonumber \\
  & \quad  -  \frac{3}{2} \int_{\mathbb{R}^N} n_\gamma (t) |\na p_\gamma(t)|^2\big(\Delta p_\gamma(t) +G(p_\gamma(t))\big).
\end{align*}
In this form, we recognize an energy equality, where the energy involves the usual kinetic energy and a potential energy.
But  the right hand side cannot be made negative using the Cauchy-Schwarz inequality, neither directly controled.
However, after some integrations by parts it can be rewritten as
\begin{equation} \label{eq:energybis}
\begin{array}{l}
 \displaystyle \frac{d}{dt}\displaystyle \int_{\mathbb{R}^N} \left(n_\gamma(t)\, \frac{|\na p_\gamma(t)|^2}{2}  -  H^{({\frac 1 \gamma})}(p_\gamma(t))\right)
\\[10pt]
\quad\displaystyle = - \frac{\gamma^2}{1+\gamma}    \int_{\mathbb{R}^N}   n_\gamma(t) p_\gamma(t) \big(\Delta p_\gamma(t)+G(p_\gamma(t))\big)^2
-\frac{\gamma}{1+\gamma} \sum_{i,j} \int_{\mathbb{R}^N}   n_\gamma(t) p_\gamma(t) \big(\partial^2_{ij} p_\gamma(t)\big)^2
\\[10pt]
\quad\quad  +  \displaystyle \int_{\mathbb{R}^N} n_\gamma (t) |\na p_\gamma(t)|^2\left( \frac{1}{2} G(p_\gamma(t)) + \frac{2 \gamma}{\gamma +1} p_\gamma(t) G'(p_\gamma(t))  \right)\\[10pt] \quad\quad\displaystyle-\frac{\gamma}{1+\gamma}  \int_{\mathbb{R}^N}   n_\gamma(t) p_\gamma(t) (G(p_\gamma(t)))^2.
\end{array}\end{equation}
The advantage compared to the expression before is that, $p_\gamma(t)$ being bounded in $L^\infty$, it gives a self-contained control of the kinetic energy (locally in time) thanks to the Gronwall lemma, assuming the kinetic energy is initially bounded. As a consequence all the terms can be controlled as usual and, back to the initial energy dissipation form, we conclude an energy  control  on the term
\[
\frac{1}{2\gamma}  \int_0^T  \int_{\mathbb{R}^N} \frac{n_\gamma(t)}{ p_\gamma(t)} |\na p_\gamma(t)|^4  \leq C(T)
\]
which is exactly compatible with the  Lipschitz estimate on $p_\gamma$ near the free boundary.
\section{Further regularity of $w(t)$ and $\Omega(t)$}
\label{sect:regularity.w.Omega}
\setcounter{equation}{0}

In this section, we show that $w(t)$ given in~\eqref{eq:bw} solves an obstacle problem for all $t>0$,  and prove Theorem~\ref{thm:main2}.

\subsection{An obstacle problem formulation}
Since $p\in L^\infty_{\rm loc}((0,\infty);H^1(\mathbb{R}^N))$, see Lemma~\ref{lm:H1},  $ w \in C^{0,1}((0,\infty);H^1(\mathbb{R}^N))$.
We now show that for all $t>0$, the function  $w(t)$ is the unique solution of an appropriate obstacle problem.
\begin{theorem}\label{thm:obstacle}
Define the function
\begin{equation}\label{eq:FFF}
F(t) = \textrm{e}^{-G(0)t} - n^0 + \int_0^t \textrm{e}^{-G(0)s} \big(G(0)-G(p(s))\big)\, ds, \quad t>0.
\end{equation}
Then $F\in C([0,\infty); L^\infty(\mathbb{R}^N))$
and, for all $t> 0$,
$w (t)$ is the unique function in $H^1(\mathbb{R}^N)$ satisfying
\begin{equation}\label{eq:obstacle0}
w(t)\geq 0, \qquad -\Delta w(t)+ F(t)\geq 0,\qquad
w(t)( -\Delta w(t)+ F(t)) = 0.
\end{equation}
In particular $  \Delta w(t) = F(t)\chi_{\{w(t)>0\}}$.
It is also the unique solution of the obstacle problem
\begin{equation}\label{eq:obstacle1}
u\in K=\{v\in H^1(\mathbb{R}^N)\,;\, v\geq 0\}, \quad J_t(u) =\inf_{v\in K} J_t(v), \quad J_t(v)=\int_{\mathbb{R}^N} \left(\frac{1}{2} |\na v|^2  + v F(t)\right).
\end{equation}
\end{theorem}
\begin{proof}
First of all, the fact that $n^0\in L^\infty(\mathbb{R}^N)$ and the inequality $|G(0)-G(p)|\leq C p\leq Cp_M$ clearly give $F\in C([0,\infty); L^\infty(\mathbb{R}^N))$.

Next, it is obvious that $ w(t)\geq 0$ in $\mathbb{R}^N$, and \eqref{eq:Deltabw} together with the fact that $n(t)\leq 1$ in $\mathbb{R}^N$ give  $\tilde F\le F$ (see \eqref{eq:Deltabw}), and hence
$$
\Delta  w(t)  \leq F(t).
$$
It  remains to show that equality holds in the open set $\widetilde\Omega(t)=\{ w (t)>0\}$.

Using \eqref{eq:Deltabw}, and  since we already know (Proposition \ref{prop:n}) that $n=1$ a.e. in $\widetilde\Omega(t)$, it is enough to show that $F(t)=\tilde F(t)$ a.e. in $\widetilde\Omega(t)$, i.e.,
$$
\int_0^t \textrm{e}^{-G(0)s} n(s)\big(G(0)-G(p(s))\big)\, ds = \int_0^t \textrm{e}^{-G(0)s} \big(G(0)-G(p(s))\big)\, ds\quad\text{a.e.~in }\widetilde\Omega(t).
$$
In fact, it is readily seen that this equality holds for all $x\in \mathbb{R}^N$.
Indeed, we have
$$
n(x,s)\big(G(0)-G(p(x,s))\big) =  G(0)-G(p(x,s))\quad\text{for all }x\in\mathbb{R}^N,
$$
since either $p(x,s)=0$, and then both sides of this equality are zero, or $p(x,s)>0$, in which case $n(x,s)=1$ and we have equality as well.

The uniqueness of the solution of \eqref{eq:obstacle0} and the relation to the classical obstacle problem formulation \eqref{eq:obstacle1} can now be shown easily, for instance by first proving that $w(t)$ is the unique solution of the variational inequality
$$
w\in K, \qquad
\displaystyle \int_{\mathbb{R}^N} \na w \cdot\na (v-w) + \int_{\mathbb{R}^N} F(t) (v-w) \geq 0 \qquad \forall v\in K.
$$

\end{proof}


\subsection{Further regularity of $w(t)$}

Theorem \ref{thm:obstacle} allows us to improve the result of Corollary
\ref{cor:preg}.
Indeed, it is well known that the optimal regularity for the obstacle problem is $C^{1,1}$; see \cite{ALS,B,C}.
However, to prove this we need a little bit better than $F(t)\in L^\infty(\mathbb{R}^N)$.
\begin{proposition}[\cite{ALS,B}]\label{eq:optreg}
Let $v\in H^1(\mathbb{R}^N)$ be the solution of the obstacle problem
$$
\Delta   v = f \chi_{ \{v >0\}}.
$$
If the function $f$ is Dini  continuous\footnote{We recall that a function is Dini continuous if it has a modulus of continuity $\sigma(r)$ such that $\int_0^1 \frac{\sigma(r)}{r}\, dr <\infty$.} in $B_r(x_0)$ for some $x_0$ and $r>0$, then $v\in C^{1,1}(B_{r/2}(x_0))$.
\end{proposition}

Instead of the Dini continuity, we can require that the solution of $\Delta \vphi=f$ be $C^{1,1}$; see~\cite{C}.

In order to use this result, we first need to establish further regularity for the function $F(t)$, which we do now, using the bound~\eqref{eq:alpha1HS} with $\alpha =1$.
\begin{lemma}\label{lem:Freg}
The function $F$ defined in \eqref{eq:FFF} satisfies
$$
\int_{\mathbb{R}^N\setminus \supp (n^0)} |\na F(t)|^4  \leq C(T)\quad\text{for all }t\in[0,T].
$$

In particular, if $N\leq 3$ then $F(t)\in C^\alpha(\mathbb{R}^N\setminus \supp n^0)$, $\alpha=1-(N/4)$, and
$$
\|F(t)\|_{C^\alpha(\mathbb{R}^N\setminus \supp n^0)} \leq C(T) \quad\text{for all }t\in [0,T].
$$
\end{lemma}
It will be clear below that $F$ cannot be more regular than $n^0$, which is typically discontinuous (characteristic function of a set).
This explains why the estimate is restricted to the set $\mathbb{R}^N\setminus \supp(n^0)$.
However, this restriction can be lifted if $n^0$ is assumed to be more regular.

\begin{proof}
Differentiating \eqref{eq:FFF}, we get
$$
\na F(t) =  - \int_0^t \textrm{e}^{-G(0)s} G'(p(s)) \na p(s)\, ds\quad\text{in }
\mathbb{R}^N\setminus \supp n^0.
$$
In view of \eqref{eq:G}, there exists a constant $C$ such that $|G'(p)|\leq C$ for $p\in [0,p_M]$, and so
\begin{align*}
|\na F(t)| & \leq   C \int_0^t \textrm{e}^{-G(0)s} |\na p(s)|\, ds \\
& \leq  C \left(\int_0^t \textrm{e}^{-\frac{4}{3} G(0)s}\, ds\right)^{3/4} \left( \int_0^t |\na p(s)|^4\, ds\right)^{1/4}\quad\text{in }
\mathbb{R}^N\setminus \supp n^0.
\end{align*}
Hence,
$$
\int_{\mathbb{R}^N\setminus \supp (n^0)} |\na F(t)|^4  \leq C \left(\int_0^t \textrm{e}^{-\frac{4}{3} G(0)s}\, ds\right)^{3}  \int_0^t\int_{\mathbb{R}^N\setminus \supp (n^0)} |\na p|^4.
$$
Therefore, for a different constant $C$,
$$
\int_{\mathbb{R}^N\setminus \supp (n^0)} |\na F(t)|^4 \leq C \int_0^T\int_{\mathbb{R}^N}  |\na p|^4\quad\text{for all }t\leq T,
$$
which combined with \eqref{eq:alpha1HS} with $\alpha=1$ yields the result.
\end{proof}

Using Lemma \ref{lem:Freg}, we can now apply Proposition \ref{eq:optreg} to get the optimal regularity for $w(t)$.
\begin{corollary}
The function $w(t)$ belongs to $C^{1,1}(\mathbb{R}^N\setminus \supp n^0)$  for all $t>0$.
\end{corollary}

\subsection{The free boundary $\pa \Omega(t)$ - Proof of Theorem \ref{thm:main2}}
Finally, Theorem \ref{thm:obstacle} can also  be used to investigate the properties of the interface
$\pa \Omega(t)$,  and to prove Theorem \ref{thm:main2}.

First, we note that the set $\widetilde \Omega(t)=\{w(t)>0\}$ is the non-coincidence set
for an obstacle problem.
Classically, this implies some regularity for the free boundary $\pa \widetilde \Omega(t)$
in the neighborhood of a point $x\in\partial\Omega(t)$,
provided that $F$ satisfies the non-degeneracy condition $F(t)>0$ in a neighborhood of $x$.

Since $G(0)\ge G(p)$ for all $p\ge 0$, we have
$$
F(t) \geq \textrm{e}^{-G(0)t} - n^0 = \textrm{e}^{-G(0)t}\big(1-n^0\textrm{e}^{G(0)t}\big).
$$
So the non-degeneracy condition is satisfied at $x$ as soon as
$$
n^0 \textrm{e}^{G(0)t}<1 \quad\text{in a neighborhood of } x,
$$
a condition which is in particular satisfied outside the support of $n^0$.

On the other hand, \eqref{eq:Omegabw} together with
condition \eqref{eq:ndg} imply that
$$
\widetilde\Omega(t) \cap B_\delta(x) = \Omega(t) \cap B_\delta(x).
$$
Hence, Theorem~\ref{thm:main2} will follow from the results proved in \cite{C0} for the solution of $\Delta v=f\chi_{\{v>0\}}$. In that paper $f$ is asked to be a positive Lipschitz function. However, as noted in \cite{B}, all that is actually needed for the proof, once $f>0$,  is that $f\in W^{1,p}(\mathbb{R}^N)$ for some $p>N$. This regularity follows from
Lemma \ref{lem:Freg} when $N\le3$.

\section{Hele-Shaw free boundary condition}
\label{sect:velocity.free.boundary}
\setcounter{equation}{0}

The goal of this section it to prove
Theorem \ref{thm:main3}, which gives the velocity of the free boundary in terms of the gradient of the pressure and the initial data. We will prove that the free boundary condition~\eqref{eq:Stefan.condition} is satisfied both in a weak (distributional) sense and in a measure theoretical sense. The starting point is to write equation~\eqref{eq:main} in an equivalent form in which the free boundary condition becomes apparent.

\noindent\emph{Remark. } As mentioned in the introduction, it is also possible to
show that the free boundary condition holds in a viscosity sense, see  \cite{KP}.

\subsection{An equivalent weak formulation}

The key point is that, thanks to Theorem~\ref{thm:main1}, we have
\begin{equation}
\label{eq:characterization.n}
n(t) = \chi_{\Omega(t)} + \textrm{e}^{G(0)t} n^0 (1-\chi_{\Omega(t)}) \quad  \text{in }Q;
\end{equation}
(see formula~\eqref{eq:defnb}).
Hence, on the one hand  we have, using that $p(t)=0$ in $\RR^N\setminus\Omega(t)$,
$$
n(t) G(p(t)) =  G(p(t))\chi_{\Omega(t)} + G(0) \textrm{e}^{G(0)t} n^0 (1-\chi_{\Omega(t)}) \quad \text{in }Q,
$$
and on the other hand,
$$
\pa_t n(t) = (1- \textrm{e}^{G(0)t} n^0 ) \pa_t \chi_{\Omega(t)} + G(0) \textrm{e}^{G(0)t} n^0 (1-\chi_{\Omega(t)}) \quad \text{in } \mathcal D'(Q).
$$
Combining these two equations with  \eqref{eq:main}, we obtain
\eqref{eq:hsweak0}:
\begin{equation}\label{eq:hsweak}
 (1- \textrm{e}^{G(0)t} n^0 )  \pa_t \chi_{\Omega(t)}   = \Delta p(t) + G(p(t))\chi_{\Omega(t)} \quad \text{in } \mathcal D'(Q).
\end{equation}
Formally at least, this equation
is a weak formulation of the following Hele-Shaw free boundary problem
\begin{equation}
\label{eq:new.weak.formulation}
\displaystyle\Delta p(t) + G(p(t)) = 0\quad\text{in } \Omega(t)=\{p(t)>0\}, \qquad
\displaystyle\pa_t p (t)=  \frac{|\na p(t)|^2}{1-\textrm{e}^{G(0)t}n^0}  \quad\text{on } \pa \Omega(t).
\end{equation}
The first equation is obvious, and in the rest of this section, we will show that \eqref{eq:hsweak} indeed implies that the free boundary condition holds in two different senses (weak and measure theoretic).

\subsection{Weak formulation of Hele-Shaw free boundary condition}

We first check that weak solutions to~\eqref{eq:hsweak} satisfy the free boundary condition in~\eqref{eq:new.weak.formulation} in a weak sense.
\begin{proposition}
Assume that $p$  solves \eqref{eq:hsweak} and is continuous (e. g. for $N\leq 3$). Then for all test functions $\vphi\in\mathcal D(Q)$, we have:
\begin{equation}\label{eq:beta}
 \lim_{\epsilon, \delta\to 0^+ }  \int_0^\infty  \frac 1 \epsilon \left(\int_{\{ \delta < p(t) < \delta + \epsilon\}} \left( \left(1-\mbox{\rm e}^{G(0)t}n^0\right){\pa_t p(t)} - |\na p(t)|^2 \right) \vphi(t)\right)\, dt = 0.
\end{equation}
\end{proposition}

Notice that  the first term of this formula makes sense only after integration by parts in time because $\pa_t p(t)$ is only defined as a nonnegative measure. With the notations below it should be written as 
\[- \int_0^\infty  \int \beta_{\epsilon, \delta}(t)  \partial_t \big( (1-e^{G(0)t} n^0)  \vphi(t) \big) dx dt 
\]
Using  that $p$ is smooth in $\overline {\{p(t)>\delta\}}$ for all $\delta>0$, from \eqref{eq:pdelta2}, assuming $p$ is continuous (e.g. in dimension less that 3), and using the coarea formula, we can then infer that
$$
 \lim_{\delta\to 0^+ }  \int_0^\infty \left(\int_{\{  p(t)= \delta \}} \left( \left(1-\textrm{e}^{ G(0)t}n^0\right)\frac{\pa_t p(t)} { |\na p(t)| }- |\na p(t)| \right) \vphi(t)\, d\mathcal H^{n-1}(x)\right)\,  dt = 0,
$$
which is the Hele-Shaw free boundary condition in the sense of distributions, namely (assuming $p$ is regular enough, for instance for $N\leq 3$ and large times as we prove it in  Section~\ref{sect:eventual.regularity})
$$
\int_0^\infty\left( \int_{\pa\Omega(t)} \left( \left(1-\textrm{e}^{G(0)t}n^0\right) \frac{\pa_t p(t)}{|\na p(t)|} - |\na p(t)|\right) \vphi(t) \, d\mathcal H^{n-1}(x)\right)\, dt  = 0.
$$

\begin{proof}
Let $b_{\epsilon, \delta}(s)$ be the piecewise linear approximation of $ \chi_{\mathbb{R}_+}(s)$ given by
$$
b_{\epsilon, \delta}(s) = \left\{
\begin{array}{ll}
0 \qquad & \text{for }  s \leq \delta,
\\
\frac {s-\delta}{\epsilon} \qquad & \text{for }  \delta \leq s \leq \delta + \epsilon,
\\ 1  \qquad & \text{for } s \geq \delta + \epsilon,
\end{array} \right.
$$
and  $\beta_{\epsilon, \delta}(s)=b'_{\epsilon, \delta}(s)$. Let us  denote $h(x,t)=\textrm{e}^{G(0)t}n^0(x)$.
Then, for any test function $\vphi\in\mathcal{D}(Q)$,
\begin{equation}
\label{eq:limt}
\displaystyle\iint_Q h \pa_t b_{\epsilon, \delta}(p)  \,\vphi= \iint_Q h\beta_{\epsilon, \delta}(p)\pa_t p\,  \vphi
= \int_0^\infty \frac 1 \epsilon\left( \int_{\{\delta < p(t) <\delta + \epsilon \}} h(t) {\pa_t p(t)} \,\vphi(t) \right) dt.
\end{equation}

Similarly, introducing $B_{\epsilon, \delta}(s) = \int_0^s b_{\epsilon, \delta}(r)\, dr$, we can write, using the fact that $\Delta p + G(p)=0$  in $\{p>\delta\}$ because $p$ is assumed to be continuous,
\begin{equation}
\label{eq:limd}
\begin{array}{rcl}
\displaystyle\iint_Q\big( \Delta(B_{\epsilon, \delta}(p)) + G(p)b_{\epsilon, \delta} (p) \big)\vphi & =&\displaystyle
\iint_Q \Big(\beta_{\epsilon, \delta}(p) |\na p|^2 +  b_{\epsilon, \delta}(p)\big(\Delta p+G(p)\big)\Big) \vphi \\[10pt]
& =& \displaystyle\iint_Q\beta_{\epsilon, \delta}(p) |\na p|^2  \vphi\\[10pt]
& =& \displaystyle\int_0^\infty \frac 1 \epsilon \left(\int_{\{\delta < p(t) <\delta + \epsilon \}}  |\na p(t)|^2  \vphi(t)\right)\, dt.
\end{array}
\end{equation}
Finally, we note that when $\delta \to 0$ and $\epsilon \to 0$,  we have (for any non-negative function $p$):
$$
b_{\epsilon, \delta}  (p) \to \chi_{\{p>0\}}, \qquad B_{\epsilon, \delta}  (p) \to p,
$$
where the convergence holds pointwise, and thus (by the Lebesgue dominated convergence theorem) in $L_{\rm loc}^1(Q)$.
In particular, we have
$$
\begin{array}{l}
\displaystyle
\left\langle h  \pa_t b_{\epsilon, \delta} (p), \vphi\right\rangle_{\mathcal D'(Q),\mathcal D(Q)} \to
\left\langle h  \pa_t \chi_{\{p>0\}}, \vphi\right\rangle_{\mathcal D'(Q),\mathcal D(Q)},
\\[10pt]
\displaystyle
\left\langle  \Delta(B_{\epsilon, \delta}(p)) + G(p)b_{\epsilon, \delta}  (p) , \vphi\right\rangle_{\mathcal D'(\mathbb{R}^N),\mathcal D(\mathbb{R}^N)} \to
\left\langle   \Delta p + G(p)\chi_{\{p>0\}} , \vphi\right\rangle_{\mathcal D'(Q),\mathcal D(Q)}.
\end{array}
$$
Therefore, using the weak formulation of \eqref{eq:hsweak} together with  \eqref{eq:limt} and \eqref{eq:limd}, we deduce \eqref{eq:beta}.
\end{proof}

\subsection{Measure theoretic formulation of Hele-Shaw free boundary condition}
Since we proved (under the additional assumptions of Theorem \ref{thm:main2}) that the free boundary has finite $\mathcal H^{n-1}$ measure, we can also interpret the weak formulation of \eqref{eq:hsweak}
in a more measure theoretical sense.
First, we prove:
\begin{lemma}\label{lem:mu}
Assume that $N\leq 3$ and that $n^0 e^{G(0)t}<1$ in a neighborhood of $\pa\Omega(t)$.
Then, the distribution  $\mu_t = \Delta p(t) + G(p(t))\chi_{\Omega(t)}$ is a non-negative Radon measure supported on $\pa \Omega(t)$ for all $t>0$.
\end{lemma}
Formally, the measure $\mu_t$ is actually $|\na p(t)|\, \mathcal H^{n-1}|_{\pa\Omega(t)}$, while $\pa_t  \chi_{\Omega(t)}=
V_n\;  \mathcal H^{n-1}|_{\pa \Omega(t)}$.
So we can interpret \eqref{eq:hsweak} as
$$
V_n\;  \mathcal H^{n-1}|_{\pa \Omega(t)} 
 =\frac{|\na p(t)|}{1- \textrm{e}^{G(0)t} n^0}\,  \mathcal H^{n-1}|_{\pa\Omega(t)},
 $$
which is the Hele-Shaw free boundary condition written in measure theoretical notations.
\begin{proof}[Proof of Lemma \ref{lem:mu}]
We recall that the support of a distribution $\mu_t$ is defined as
$$
\supp \mu_t = \{x\in \RR^N\,;\, \forall r>0 \; \exists \vphi \in \mathcal D(B_r(x)) \mbox{ such that } \mu_t(\vphi)\neq 0\}.
$$
With that definition in mind, we note that if $\vphi\in\mathcal{D}(\mathbb{R})$ is such that $\supp \vphi \subset \mbox{Int}(\Omega(t))$, then \eqref{eq:pdelta2}  implies
$$
\mu_t (\vphi) = \langle\Delta p(t) + G(p(t)), \vphi \rangle_{\mathcal D'(\RR^N),\mathcal D(\RR^N)} = 0,
$$
while if  $\vphi$ is such that $\supp \vphi \subset \RR^N\setminus \overline {\Omega(t)}$ then, by the definition of $\Omega(t)$,
$$
\mu_t (\vphi) = \langle\Delta p(t) , \vphi \rangle_{\mathcal D'(\RR^N),\mathcal D(\RR^N)} =
\int_{\mathbb{R}^N} p(t) \Delta \vphi = 0.
$$
We deduce that $\supp \mu_t \subset \pa \Omega(t)$.

Next, we fix $\vphi\in \mathcal D(\RR^N)$ such that $\vphi\geq 0$ and we consider a sequence $\psi_k\in C^\infty (\RR^N)$ such that $0\leq \psi_k(x)\leq 1$ for all $x\in\RR^N$,
$ \psi_k(x)=1$ in a neighborhood of $\pa\Omega(t)$ and $\int_{\mathbb{R}^N} \psi_k \tto 0$ as $k\to \infty$
(such a sequence exists because $\pa\Omega(t)$ has Hausdorff dimension $N-1$
by Theorem \ref{thm:main2}).
Since the function $ (\psi_k-1)\vphi$ is supported in $\Omega(t)\cup  \big(\RR^N\setminus \overline {\Omega(t)}\big)$, the first part of the proof implies
$$ \mu_t(\vphi \psi_k ) = \mu_t (\vphi) + \mu_t((\psi_k-1)\vphi) = \mu_t(\vphi).$$
Using \eqref{eq:pdelta1} and the monotonicity of $G$, we thus have
\begin{align*}
\mu_t(\vphi)= \mu_t(\vphi \psi_k ) & = \langle\Delta p(t) + G(p(t))\chi_{\Omega(t)}, \vphi \psi_k\rangle_{\mathcal D'(\RR^N),\mathcal D(\RR^N)}  \\
&  = \langle\Delta p(t) + G(p(t)), \vphi \psi_k\rangle_{\mathcal D'(\RR^N),\mathcal D(\RR^N)}
-  \langle G(p(t))(1-\chi_{\Omega(t)} ), \vphi \psi_k\rangle_{\mathcal D'(\RR^N),\mathcal D(\RR^N)} \\
& \geq  -  \int_{\RR^N} G(0)(1-\chi_{\Omega(t)}) \vphi\psi_k\geq - G(0)\|\vphi\|_{L^\infty} \int_{\RR^N} \psi_k,
\end{align*}
where this last term goes to zero as $k\to \infty$. We deduce that
$$
\mu_t(\vphi) \geq 0 \quad \mbox{for all } \vphi \in \mathcal D'(\RR^N), \;\vphi\geq 0,
$$
which completes the proof (since non-negative distributions are in fact Radon measures).
\end{proof}



\section{Eventual regularity of solutions}
\label{sect:eventual.regularity}
\setcounter{equation}{0}

In general one cannot expect the free boundary $\partial \Omega(t)$ to be smooth, due to the collision between parts of it and the generation of new islands. However, we will prove that it  becomes smooth for large times.

In the case with no reaction, $G\equiv0$, the regularity of the solution and its free boundary for large times was proved in \cite{Kim-2006}, following closely what is done for the Porous Medium Equation in \cite{Caffarelli-Vazquez-Wolanski-1987}. Here we will take profit of the relation between the Hele-Shaw problem and an obstacle problem to make things a bit simpler. We will show that each point on $\partial\Omega(t)$ has positive Lebesgue density with respect to the healthy region $\mathbb{R}^N\setminus\Omega(t)$ if $t$ is sufficiently large. Then, smoothness will follow from Caffarelli's regularity results for the free boundary of the obstacle problem \cite{Caffarelli-1977}, whenever $\partial\Omega(t)$ does not intersect the support of $n^0$, which guarantees the required non-degeneracy condition. Our main tool is Alexandroff's Reflection Principle, which exploits that equation~\eqref{eq:pme} is invariant under plane symmetries. A similar approach was used in~\cite{Matano-1982} to study the long time regularization of the solution and the free boundary of the Stefan Problem.

Let $n_\gamma$ be a solution to~\eqref{eq:pme} with initial data supported in the ball $B_R(0)$, $R>0$.   Let $H$ be an $(N-1)$-dimensional hyperplane such that $H\cap B_R(0)=\emptyset$. It divides the space $\mathbb{R}^N$ into two half-spaces: one, $S_1$, which contains the support of $n_\gamma^0$, and another one $S_2$ such that $n_\gamma^0=0$ there. Let $\Pi:S_1\to S_2$ be the specular symmetry with respect to the hyperplane $H$. An easy comparison argument shows that
$$
n_\gamma(x,t)\ge n_\gamma(\Pi(x),t),\qquad x\in S_1.
$$
An analogous result is valid for the pressure $p_\gamma$. By letting $\gamma\to\infty$, we obtain a similar result for solutions of~\eqref{eq:main}--\eqref{eq:initial.data} obtained by approximation by solutions to~\eqref{eq:pme}.

\begin{lemma}
\label{lem:Alexandroff}
Let $(n,p)$ be a stiff limit solution to problem~\eqref{eq:main}, corresponding to a sequence of approximating initial data such that $\mathop{supp\,}n^0_\gamma\subset B_R(0)$ for some $R>0$. Then,  for all $(N-1)$-dimensional hyperplane $H$ such that $H\cap B_R(0)=\emptyset$, we have
$$
n(x,t)\ge n(\Pi(x),t),\quad p(x,t)\ge p(\Pi(x),t), \qquad\text{for a.e. } x\in S_1,
$$
where $S_1$ is the half-space determined by $H$ that contains $B_R(0)$ and $\Pi$ is the specular symmetry with respect to $H$.
\end{lemma}

The next step is to prove that the free boundary goes to infinity.
\begin{lemma}
\label{lem:boundary.goes.away}
Under the hypotheses of Theorem~\ref{thm:main1}, if the initial data satisfy $n^0\ge \alpha\chi_{B_r(\bar x)}$ for some $\alpha\in(0,1]$, $r>0$, and $\bar x\in\mathbb{R}^N$, then
$$
\lim_{t\to\infty}\mathop{\rm dist\,}(\partial\Omega(t),0)=\infty.
$$
\end{lemma}

\begin{proof}
Compare with the solution with initial data $\chi_{B_r(\bar x)}$, which after a finite time becomes a tumor spheroid. Tumor spheroids were studied in~\cite{PQV}.
\end{proof}

\begin{theorem}
\label{th:regFB}
Let $N\le 3$, and assume the hypotheses of Lemmas~\ref{lem:Alexandroff}--\ref{lem:boundary.goes.away}. There exists a time $T_0$ such that  $\partial\Omega(t)$ is starshaped and smooth for all $t\ge T_0$.
\end{theorem}

\begin{proof}
By virtue of  Lemma~\ref{lem:boundary.goes.away}, there exists a positive time $T_0$ such that
$\partial\Omega(t)\cap B_R(0)=\emptyset$ for all $t\ge T_0$.

Take any $t_0\ge T_0$ and fix it. Let $x_0$ any point on $\partial\Omega(t_0)$. We must check that $x_0$ has positive Lebesgue density with respect to the healthy region $\mathbb{R}^N\setminus\Omega(t_0)$.  To this aim we consider the non-empty open cone with vertex in $x_0$ given by
$$
K_0=\{x\in\mathbb{R}^N: (x-x_0)\cdot(y-x_0)<0\text{ for all }y\in \overline{B_R(0)} \}.
$$
Let $x$ be any point in $K_0$ and set
$$
H_\alpha=\{y\in\mathbb{R}^N: (y-x_0-\alpha\xi)\cdot\xi=0\}, \qquad \xi=x-x_0.
$$
Applying Lemma~\ref{lem:Alexandroff} to $H=H_\alpha$ for each $\alpha\ge0$, we easily find that $p(x_0+\alpha\xi,t_0)$ is monotone non-increasing in $\alpha\ge0$. Consequently,
$$
p(x_0+\alpha\xi,t_0)\le p(x_0,t_0)=0\quad\text{for all }\alpha\ge0.
$$
This shows that $\partial\Omega(t)$ is starshaped with respect to the origin and that $p$ vanishes everywhere in $K_0$. Hence $x_0$ has positive Lebesgue density with respect to $\mathbb{R}^N\setminus \Omega(t)$. The conclusion follows by applying the regularity results for the Obstacle Problem from~\cite{Caffarelli-1977}.
\end{proof}

\noindent\emph{Remark. } Let us write the pressure in polar coordinates, $p=p(r,x/|x|,t)$. We have actually proved that $p$ is monotone non-increasing in $r$ outside the ball $B_R(0)$. In particular, the maximum of $p(t)$ is attained in this ball.

\medskip

Let $R_+(t)=\inf\{r>0: \Omega(t)\subset B_r(0)\}$, $R_-(t)=\sup\{r>0: B_r(0)\subset\Omega(t)\}$. As a second,  easy corollary of Alexandroff's Reflection Principle, we get a control of the
difference between these two quantities.
\begin{corollary}
$R_+(t)-R_-(t)\le 2R$.
\end{corollary}

In order to finish the proof of Theorem~\ref{thm:main4}, we still have to prove the eventual continuity of the pressure, which is done next.
\begin{theorem}
With the hypotheses of Theorem~\ref{th:regFB},  the pressure $p$  is continuous in space and time for all $t\ge T_0$.
\end{theorem}

\begin{proof}
We know from the estimates in~\cite{PQV} that $p(t)\in H^1(\mathbb{R}^N)$ for all $t>0$. Since the free boundary is smooth for $t\ge T_0$, we have then that $p(t)\in H_0^1(\Omega(t))$, and because of standard regularity results for elliptic equations in smooth domains, $p(t)$ is moreover Lipschitz.

From the regularity in space, we will now deduce that $p$ is Lipschitz continuous in time.
Consider $s >T_0$.
From the proof of Theorem \ref{th:regFB} above, we see that there exists $\eta$ such that  for all $x_0 \in \pa\Omega(s)$,
\begin{equation}\label{eq:density3}
 |B_r(x_0) \setminus \Omega(s) |\geq \eta r^{N}.
\end{equation}
Furthermore, the weak formulation of \eqref{eq:hsweak} yields
\begin{align*}
\frac{d}{dt} |B_r(x_0)\cap \Omega(t)|& \leq \int_{\pa B_r(x_0)} \na p\cdot \nu d\sigma(x) + G(0)|B_r(x_0)\cap\Omega(t)|\\
& \leq Cr^{N-1} + G(0)|B_r(x_0)\cap\Omega(t)|
\end{align*}
where we used the uniform bound on $|\na p|$ (Lipschitz regularity in space).
A Gronwall argument gives
$$
|B_{r}(x_0)\cap\Omega(t)|\leq |B_{r}(x_0)\cap\Omega(s)|\textrm{e}^{G(0)(t-s)} + Cr^{N -1}(e^{G(0)(t-s)}-1).
$$
In particular, when $0\leq t-s\leq 1$, the density estimate \eqref{eq:density3} (and the fact that the free boundary has zero Lebesgue measure) yields
\begin{align*}
 |B_r(x_0) \setminus \Omega(t) |& =|B_r| -  |B_{r}(x_0)\cap\Omega(t)| \\
& \geq  |B_r(x_0) \setminus \Omega(s) | -|B_r(x_0)\cap\Omega(t)|\left(\textrm{e}^{G(0)(t-s)}-1\right) -Cr^{N-1}\left(\textrm{e}^{G(0)(t-s)}-1\right)\\
& \geq r^{N-1}\Big((\eta-C_1(t-s)) r-C_2(t-s)\Big)
\end{align*}
for some constants $C_1$, $C_2$ (depending only on $G(0)$ and the Lipschitz constant of $p$).
Therefore, if $t-s$ is small enough so that $(\eta-C_1(t-s))\ge \eta/2$, and we take $C_0=4C_2/\eta$, we deduce that $B_{C_0(t-s)} (x_0) \setminus \Omega(t)\neq \emptyset$. Then,  the Lipschitz regularity of $p$ in space implies $p(x_0,t) \leq K(t-s)$ for all $t\in(s,s+\varepsilon)$, for some small $\varepsilon>0$, and for all $x_0\in\pa\Omega(s)$. The continuity in time for points $x\in \mathbb{R}^N\setminus\Omega(s)$ then easily follows from the monotonicity in the radial direction away from~$B_R(0)$.

On the other hand, the function $q(x):= p(x,t) -K(t-s)$ satisfies $q(x) \leq 0$ on $\partial \Omega(s)$. Moreover, because $G$ is decreasing,
$$
- \Delta q =G\big(q+K(t-s) \big) \leq G(q)\quad\text{in }\Omega(s).
$$
We immediately conclude that $q \leq p$ in $\Omega(s)$, which means, using also the monotonicity in time of $p$, that $0 \leq p(x,t)- p(x,s) \leq K(t-s)$ in that set.
\end{proof}

\section{Convergence of the positivity sets}
\label{convergence.positivity.sets}
\setcounter{equation}{0}

The aim of this section is to prove that the positivity sets of the functions $n_\gamma(t)$ converge, in the sense of Hausdorff distance, as $\gamma\to\infty$ to the positivity set of $n(t)$, Theorem~\ref{thm:convergence.positivity.sets}. In the case with no reaction, $G\equiv0$, this analysis was performed in \cite{GQ}.
Notice that $\{n_\gamma(t)>0\}=\{p_\gamma(t)>0\}$ for all $t>0$. However,  it is not true in general that $\{n(t)>0\}=\{p(t)>0\}$.

In order to prove this theorem we will use that $n_\gamma(t)$ converges to $n(t)$ in $L^1(\mathbb{R}^N)$ for all $t>0$.

\begin{lemma}
Let $t>0$. Then $\|n_\gamma(t)-n(t)\|_{L^1(\mathbb{R}^N)}\to 0$ as $\gamma\to\infty$.
\end{lemma}
\begin{proof}
The \lq almost contraction' property
$$
\int_{\mathbb{R}^N} \{n_\gamma(t)-\hat n_\gamma(t)\}_+\le
\textrm{e}^{G(0)t}\int_{\mathbb{R}^N} \{
n_\gamma(0)-\hat n_\gamma(0)\}_+,
$$
proved in~\cite{PQV}, together with Frechet-Kolmogorov's compactness criteria yield local convergence in $L^1(\mathbb{R}^N)$. The result then follows from the
uniform control of the supports of the functions $n_\gamma(t)$, also obtained in~\cite{PQV}.
\end{proof}

\noindent\emph{Notation. }  We denote the
$\delta$-neighbourhood of a set $A\subset\mathbb{R}^N$ by $V_\delta(A)$, i.e.,
$$
V_\delta(A)=\{x\in\mathbb{R}^N, d(x,A)<\delta\}.
$$

We have to prove that for any $\delta>0$ and $t>0$ we have
\begin{equation}
\label{eq:inclusions}
\{n(t)>0\}\subset V_\delta(\{n_\gamma(t)>0\}),\quad \{n_\gamma(t)>0\}\subset V_\delta(\{n(t)>0\})\quad\text{for all }\gamma\text{ large enough.}
\end{equation}

We begin by proving the
easy part of the convergence result.
\begin{lemma} Let $t>0$. Given any $\delta>0$,
$$
\{n(t)>0\} \subset V_\delta(\{n_\gamma(t)>0\})\quad\text{for all $\gamma$ large enough}.
$$
\end{lemma}

\begin{proof} We already know from Theorem~\ref{thm:main1} that
$\{n(t)>0\}=\Omega(t)\cup \{n^0>0\}$. On the one hand, the retention property for the Porous Medium Equation implies that  $\{n_\gamma^0>0\}\subset\{n_\gamma(t)>0\}$ for all $\gamma>0$ and $t>0$. Hence, using the convergence of the positivity sets of the initial data,
$$
\{n^0>0\}\subset V_\delta(\{n_\gamma^0>0\})\subset V_\delta(\{n_\gamma(t)>0\}).
$$
On the other hand, $n(t)= 1$ on $\Omega(t)$. Moreover, $\Omega(t)$ is bounded, see \cite{PQV}, and we
can cover it with a finite family of balls $B_i=B(x_i,\delta/2)$,
with each $x_i\in \Omega(t)$. Note that all the intersections $B_i\cap \Omega(t)$
are non void open sets, hence
$$
\min_i \|n(t)\|_{L^1(B_i)}=\eta>0.
$$
If we take $\gamma$ such that
$$
\|n_\gamma(t)-n(t)\|_{L^1(\mathbb{R}^N)}<\eta,
$$
then $n_\gamma(t)$ can not vanish identically in any of the balls $B_i$
and, as an immediate consequence,
$$
B_i \subset V_\delta(\{n_\gamma(\cdot,t)>0\}) \quad \mbox{\rm for all }
i.
$$
\end{proof}

The \lq difficult' part of the proof, the second inclusion in~\eqref{eq:inclusions},  uses comparison with certain supersolutions to
\begin{equation}
\label{eq:pressure.with.G(0)}
\partial_t p =\gamma p\Delta p + |\na p|^2 + \gamma p G(0).
\end{equation}

\begin{lemma}
\label{lemma.on.supersolutions} The radial function
\begin{equation}
\label{supersolutions} P(x,t)=
\left(C-\frac{(|x|-r_0)^2}{4t}\right)_+,\qquad |x|\le r_0
\end{equation}
is a supersolution of~\eqref{eq:pressure.with.G(0)} in
$$
|x|<r_0,\qquad 0\le t\le \min\left\{\frac1{4G(0)},\frac{r_0^2}{16N^2C}\right\}.
$$
Moreover, in that range of times $P=0$ in the ball $B_{r_0'}(0)$, $r'=\frac{r_0(2N-1)}{2N}$.
\end{lemma}
\begin{proof}
Straightforward, just work out the computations and use the
equation~\eqref{eq:pressure.with.G(0)} for the pressure.
\end{proof}

We start by proving  a strong version of the second inclusion in~\eqref{eq:inclusions} for \lq small times'.

\begin{lemma}
Given $\delta>0$, let $\delta'=\frac{\delta (2N-1)}{2N}$.
There exists a constant $\gamma_0=\gamma_0(\delta)>0$ such that $V_{\delta'}(\{n_\gamma(t)>0\})\subset V_\delta(\{n(t)>0\})$ for all $t\le 1/\gamma$ if $\gamma\ge\gamma_0$.
\end{lemma}
\begin{proof}
Let $x_0\in\mathbb{R}^N$ such that $B(x_0,\delta)\subset\mathop{\rm Int}(\{n(t)=0\})\subset\mathop{\rm Int}(\{n^0=0\})\subset V_{\delta/4}(\mathop{\rm Int}(\{n_\gamma^0=0\}))$ with $t$ as above. Since $p_\gamma\le p_M$, a comparison argument with the function~\eqref{supersolutions} centered at $x_0$ with $C=p_M$ shows that $p_\gamma=0$ in  $B\left(x_0,\delta'\right)$ if
$$
t\le \min\left\{\frac1{4G(0)},\frac{r_0^2}{16N^2p_M}\right\}=:\frac1{\gamma_0}.
$$
\end{proof}
To deal with bigger times, we need a better control of the size of $p_\gamma$, in order to be able to apply the comparison argument. Namely, we will use that $p_\gamma$  converges uniformly to $0$ on suitable subsets of the
complement of the support of $n(t)$.
\begin{lemma}
\label{uniform.convergence.for.p} Given $t>0$ and $\delta>0$,  let $x\in\mathbb{R}^N$ such that $B(x,2\delta)\subset\mathop{\rm Int} (\{n(t)=0\})$. Then, given $\epsilon>0$, there is a value $M>1$ such that
$$
p_\gamma(y,s) < \epsilon, \qquad\mbox{for all } (y,s)\in
B(x,\delta)\times [1/\gamma,t],\qquad\mbox{for all }\gamma\ge M.
$$
\end{lemma}
\begin{proof}
Let $h_\gamma=h_\gamma(y,s)$ be the solution to the problem
$$
\Delta h_\gamma(y,s) = n_\gamma(y,s), \quad y\in B(x,2\delta),
\qquad
h_\gamma(y)=0, \quad |y-x|=2\delta.
$$
Since
$$
-\Delta n_\gamma^{\gamma+1} =\frac{\gamma+1}{\gamma}\left( -\partial_t n_\gamma+n_\gamma G(p_\gamma)\right)\le \frac{\gamma+1}{\gamma}\left(\frac{c\textrm{e}^{-\gamma c t}}{1-\textrm{e}^{-\gamma c t}}+G(0)\right)n_\gamma\le K_G n_\gamma
$$
for all $\gamma\ge1$ and $t\ge 1/\gamma$,
the function
$$
n_\gamma^{\gamma+1}(y,t)+K_G h_\gamma(y,t)
$$
is subharmonic. Then, for $y\in B(x,\delta)$, $s\ge 1/\gamma$,
\begin{equation}
\label{estimacion.puntual.de.p.gamma}
n_\gamma^{\gamma+1}(y,s)
\le
\frac1{\omega_N \delta^n} \int_{B(y,\delta)} n_\gamma^{\gamma+1}(z,s) \, dz +
\frac{K_G}{\omega_N \delta^n} \int_{B(y,\delta)}
h_\gamma(z,s)\, dz
- K_G h_\gamma(y,s).
\end{equation}
In order to control the right hand side of this inequality, we
first estimate the size of $n_\gamma(s)$, for $s\in(1/\gamma,t)$. To
do so we consider an integrated version of the inequality
$$
\partial_s  n_\gamma(s) \ge - \frac{cn_\gamma(s)\textrm{e}^{-\gamma c s}}{1-\textrm{e}^{-\gamma c s}}.
$$
For $s_1<s_2$ we obtain
$$
n_\gamma(s_1)
\le
\left(\frac{1-\textrm{e}^{-\gamma c s_2}}{1-\textrm{e}^{-\gamma c s_1}}\right)^{1/\gamma c}  n_\gamma(s_2).
$$
This gives immediately
\begin{equation}
\label{estimacion.puntual.de.n.gamma}
n_\gamma(s)
\le
C  n_\gamma(t), \qquad s\in [1/\gamma,t].
\end{equation}
Since $n_\gamma(t)$ converges to $0$ in $L^1(B(x,2\delta))$, and the functions $n_\gamma(s)$ are uniformly bounded, we have that the functions $n_\gamma(s)$ converge to $0$ in $L^p(B(x,2\delta))$ for all $p\in[1,\infty)$ and $s\in[1/\gamma,t]$. Therefore, applying standard theory for the Laplace equation we get
that the norm of $h_\gamma(\cdot,s)$ in $L^\infty(B(x,2\delta))$ is
is  small, uniformly for $s$ in $[1/\gamma,t]$.
This gives a control on the size of
the terms in the right hand side
of~(\ref{estimacion.puntual.de.p.gamma}) where $h_\gamma$ appears.

To control the integral term on the right hand side we use that the pressures $p_\gamma(s)$ are uniformly (in $s$ and $\gamma$) bounded. Then
$$
n_\gamma^{\gamma+1}(s) \le M n_\gamma(s)\le M Cn_\gamma(t), \qquad s\in[1/\gamma,t].
$$

We conclude that $n_\gamma^{\gamma+1}(y,s)$ is small for $y\in B(x,2\delta)$, $s\in[1/\gamma,t]$, which implies that $p_\gamma(y,s)$ is small in that set.
\end{proof}

We have now all the elements to finish the proof of the second inclusion in~\eqref{eq:inclusions}.
\begin{lemma} Let $t>0$. Given any $\delta>0$,
$$
\{n_\gamma(t)>0\} \subset V_\delta(\{n(t)>0\})
$$
for all $\gamma$ large enough.
\end{lemma}
\begin{proof}
Let $x_0\in\mathbb{R}^N$ such that $B(x_0,\delta)\subset\{n(t)=0\}^0$. Let $k=\lceil t4G(0)\rceil$. We take $\gamma_1\ge\gamma_0$ such that for all $\gamma\ge\gamma_1$
$$
p_\gamma\le \frac{r_k^2G(0)}{(2N-1)^2} \qquad \text{in } B(x_0,\delta/2)\times[1/\gamma,t], \quad\text{where }r_j=\frac\delta2\left(\frac{2N-1}{2N}\right)^j.
$$
We perform a sequence of $k$ comparisons in balls of decreasing radius $r_j$, $j=1,\dots,k$,  with supersolutions of the form~\eqref{supersolutions} with $C=r_k^2 G(0)/(2N-1)^2$ centered at $x_0$, in time intervals of length $1/(4G(0))$, starting at time $1/\gamma$, to obtain that $p_\gamma(t)=0$ in the ball $B(x_0,r_{k+1})$.
\end{proof}

\paragraph{\bf Acknowledgements.}
Antoine Mellet is partially supported by NSF Grant DMS-1201426.
Antoine Mellet would also like to thank the Laboratoire Jacques-Louis Lions which he was visiting  as part of the \lq\lq Junior Chair\rq\rq program of the Fondation Sciences Math\'ematiques de Paris when most of this work was completed.  Fernando Quirós is partially supported by the Spanish project MTM2014-53037-P. This work was started while he was visiting the  Laboratoire Jacques-Louis Lions. He is indebted to this institution for its hospitality. Beno\^\i t Perthame thanks Kibord, ANR-13-BS01-0004 funded by the French Ministry of Research.


\begin{thebibliography}{99}
\bibitem{ALS} Andersson, J.; Lindgren, E.; Shahgholian, H. \emph{Optimal regularity for the no-sign obstacle problem}. Comm. Pure Appl. Math. 66 (2013), no.\,2, 245--262.

\bibitem{AB} Aronson, D.\,G.; B\'{e}nilan, Ph. \emph{R\'{e}gularit\'{e}
des solutions de l'\'{e}quation des milieux poreux dans $R\sp{N}$}.
C. R. Acad. Sci. Paris S\'{e}r. A-B 288 (1979), no.\,2, A103--A105.

\bibitem{B} Blank, I. \emph{Sharp results for the regularity and stability of the free boundary in the obstacle problem}. Indiana Univ. Math. J. 50 (2001), no.\,3, 1077--1112.


\bibitem{bresch} T. Colin, D. Bresch, E. Grenier, B. Ribba and O. Saut,,  Computational modeling of solid tumor growth: the avascular stage, {\it SIAM Journal of Scientific Computing} {\bf 32} (4) (2010) 2321--2344.


\bibitem{Caffarelli-1977} Caffarelli, L.\,A. \emph{The regularity of free boundaries in higher dimensions}. Acta Math. 139 (1977), no.\,3-4, 155--184.

\bibitem{C0} Caffarelli, L.\,A. \emph{A remark on the Hausdorff measure of a free boundary, and the convergence of coincidence sets}. Boll. Un. Mat. Ital. A  18   (1981), no.\,1, 109--113.

\bibitem{C} Caffarelli, L.\,A.  \emph{The obstacle problem
revisited}. J. Fourier Anal. Appl. 4 (1998), no.\,4-5, 383--402.


\bibitem{Caffarelli-Vazquez-Wolanski-1987} Caffarelli, L.\,A.; Vázquez, J.\,L.; Wolanski, N.\,I. \emph{Lipschitz continuity of solutions and interfaces of the N -dimensional porous medium equation}. Indiana Univ. Math. J.  36  (1987),  no.\,2, 373--401.


\bibitem{EJ} Elliot, C.\,M.; Janovsk\'y, V. \emph{A variational inequality
approach to Hele-Shaw flow with a moving boundary}. Proc. Roy.
Soc. Edinburgh Sect. A 88 (1981), no.\,1--2,  93--107.

\bibitem{GQ-2001} Gil, O.; Quir\'os, F. \emph{Convergence of the porous media equation to Hele-Shaw}.
Nonlinear Anal. Ser. A: Theory Methods 44 (2001), no.\,8,
1111--1131.

\bibitem{GQ} Gil, O.; Quir\'os, F. \emph{Boundary layer formation in the transition from the porous
media equation to a Hele-Shaw flow}. Ann. Inst. H. Poincar\'{e}
Anal. Non Lin\'{e}aire 20 (2003), no.\,1, 13--36.

\bibitem{Giusti} Giusti, E. \emph{Minimal surfaces and functions of bounded variation}.
Monographs in Mathematics, 80. Birkh\"auser Verlag, Basel,  1984.
%

\bibitem{Kim-2006} Kim, I.\,C. \emph{Long time regularity of solutions of the Hele-Shaw problem}. Nonlinear Anal.  64  (2006),  no.\,12, 2817--2831.

\bibitem{KM} Kim, I.\,C.; Mellet, A. \emph{Homogenization of a Hele-Shaw problem in periodic and random media}. Arch. Ration. Mech. Anal. 194
(2009), no.\,2,   507--530.

\bibitem{KP} Kim, I.\,C.; Po{\v{z}}\`ar, N. \emph{Porous medium equation to Hele-Shaw flow with general initial density}. Preprint (2015) \href{http://arxiv.org/abs/1509.06287}{arXiv: 1509.06287}.

\bibitem{Lowengrub_survey}   Lowengrub, J.~S.;  Frieboes H.~B.;  Jin, F.;  Chuang, Y.-L.;  Li, X.;
 Macklin, P.;  Wise, S.~M.;  Cristini, V.  \emph{Nonlinear modelling of cancer: bridging the gap between cells and tumours}. Nonlinearity 23 (2010), no.~1, R1--R91.


\bibitem{Matano-1982} Matano, H. \emph{Asymptotic behavior of the free boundaries arising in one-phase Stefan problems in multidimensional spaces}.  Nonlinear partial differential equations in applied science (Tokyo, 1982),  133--151, North-Holland Math. Stud., 81, North-Holland, Amsterdam, 1983.

\bibitem{PQV} Perthame, B.; Quir\'os, F.;  V\'azquez, J.\,L. \emph{The Hele-Shaw asymptotics for mechanical models of tumor growth}. Arch. Ration. Mech. Anal.  212 (2014), no.\,1, 93--127.

\end{thebibliography}
\end{document}